\documentclass[12pt]{amsart}

\usepackage{framed}
\usepackage{braket}
\usepackage{amsmath,amssymb,amsthm}
\usepackage{color}
\usepackage{bm}
\usepackage[all]{xy}
\usepackage{amscd}
\usepackage{graphicx}
\usepackage{ascmac}
\usepackage{BOONDOX-frak, mathrsfs}
\usepackage[top=30truemm,bottom=30truemm,left=25truemm,right=25truemm]{geometry}
\usepackage{mathtools}
\mathtoolsset{showonlyrefs=true}
\usepackage{tikz}
\usetikzlibrary{cd}

\makeatletter
\@addtoreset{equation}{section}

\makeatother

\newcommand{\Z}{\mathbb{Z}}

\newcommand{\Q}{\mathbb{Q}}

\newcommand{\F}{\mathbb{F}}

\newcommand{\bE}{\mathbb{E}}

\newcommand{\cI}{\mathcal{I}}

\newcommand{\cL}{\mathcal{L}}

\newcommand{\wtil}[1]{\widetilde{#1}}
\newcommand{\ol}[1]{\overline{#1}}

\newcommand{\lrang}[1]{\langle #1 \rangle}

\DeclareMathOperator{\Gal}{Gal}
\DeclareMathOperator{\Coker}{Cok}
\DeclareMathOperator{\Cok}{Cok}
\DeclareMathOperator{\Ker}{Ker}
\DeclareMathOperator{\Imag}{Im}

\DeclareMathOperator{\id}{id}
\DeclareMathOperator{\Cl}{Cl}
\DeclareMathOperator{\ord}{ord}

\DeclareMathOperator{\rank}{rank}

\DeclareMathOperator{\cha}{char}

\usepackage[OT2,T1]{fontenc}
\DeclareSymbolFont{cyrletters}{OT2}{wncyr}{m}{n}
\DeclareMathSymbol{\Sha}{\mathalpha}{cyrletters}{"58}

\let\oldenumerate\enumerate
\renewcommand{\enumerate}{
   \oldenumerate
   \setlength{\itemsep}{1pt}
   \setlength{\parskip}{0pt}
   \setlength{\parsep}{0pt}
}
\let\olditemize\itemize
\renewcommand{\itemize}{
   \olditemize
   \setlength{\itemsep}{1pt}
   \setlength{\parskip}{0pt}
   \setlength{\parsep}{0pt}
}

\theoremstyle{plain}
\newtheorem{thm}{Theorem}[section]
\newtheorem{lem}[thm]{Lemma}

\newtheorem{prop}[thm]{Proposition}

\theoremstyle{definition}
\newtheorem{defn}[thm]{Definition}

\newtheorem{rem}[thm]{Remark}
\newtheorem{eg}[thm]{Example}

\usepackage{diagbox}
\usepackage{tikz}
\usetikzlibrary{shapes.geometric}
\usepackage{mathtools}
\usepackage{float}
\usepackage{booktabs}

\newcommand{\swid}[1]{\hspace{0.05em} #1 \hspace{0.05em}}
\newcommand{\sfr}{\mathsf{r}}

\newcommand{\AND}{\quad \text{and} \quad}
\DeclareMathOperator{\aug}{aug}
\DeclareMathOperator{\Pic}{Pic}
\DeclareMathOperator{\Jac}{Jac}
\DeclareMathOperator{\Div}{Div}
\DeclareMathOperator{\nt}{nt}
\DeclareMathOperator{\unr}{unr}
\DeclareMathOperator{\tors}{tors}
\DeclareMathOperator{\length}{length}
\DeclareMathOperator{\sat}{sat}

\makeatletter
\@namedef{subjclassname@2020}{\textup{2020} Mathematics Subject Classification}
\makeatother

\title
[Iwasawa-type formula for graphs]
{An Iwasawa-type asymptotic formula for multiple $\Z_p$-coverings of graphs}
\author[T.~Kataoka]{Takenori Kataoka}
\address{Department of Mathematics, Faculty of Science Division II, Tokyo University of Science.
1-3 Kagurazaka, Shinjuku-ku, Tokyo 162-8601, Japan}
\email{tkataoka@rs.tus.ac.jp}
\date{\today}

\begin{document}

\maketitle

\begin{abstract}
For a possibly ramified $\Z_p^d$-covering of connected graphs, we establish an Iwasawa-type asymptotic formula for the growth of the $p$-adic valuations of the complexities.
The formula is expressed as a polynomial in $n$ and $p^n$ with explicit leading coefficients $\lambda$ and $\mu$; in particular, we eliminate the error term of the form $O(p^{(d-1)n})$ appearing in earlier work.
We then establish a Kida-type formula describing the behavior of $\lambda$ under a $p$-covering between $\Z_p^d$-coverings, assuming $\mu = 0$.
Finally, for any fixed $p$ and integer $d \geq 2$, we construct an unramified $\Z_p^d$-covering of a bouquet with prescribed $\lambda$- and $\mu$-invariants.
\end{abstract}

\section{Introduction}\label{sec:intro}

For a connected graph $X$, let $\kappa_X$ denote the complexity of $X$, that is, the number of its spanning trees.
By Kirchhoff's matrix-tree theorem, the complexity $\kappa_X$ is equal to the order of the Jacobian group $\Jac X$.
Here and henceforth, a graph is defined as a finite undirected multi-graph, allowing loops and multi-edges.

In recent years, there has been progress in the study of analogies between number theory and graph theory.
In this analogy, a graph $X$ corresponds to a number field $K$, the Jacobian group $\Jac X$ to the class group $\Cl_K$, and the complexity $\kappa_X = \# \Jac X$ to the class number $h_K = \# \Cl_K$.

Iwasawa theory is one of the successful examples of such analogies.
Let $p$ be a prime number.
In classical Iwasawa theory, we study a $\Z_p$-extension $K_{\infty}/K$ of number fields, that is, a tower of number fields
\[
K = K_0 \subset K_1 \subset K_2 \subset \cdots
\]
such that $K_n/K$ is a $\Z/p^n \Z$-extension.
Then Iwasawa's class number formula \cite{Iwa59} asserts that there are integers $\lambda \geq 0, \mu \geq 0, \nu$ such that
\[
\ord_p(h_{K_n}) = \lambda n + \mu p^n + \nu
\]
holds for $n \gg 0$, where $\ord_p$ denotes the $p$-adic valuation normalized by $\ord_p(p) = 1$.

An analogue of Iwasawa's class number formula in graph theory is proved independently by Gonet \cite{Gon22} and McGown--Valli\`eres \cite{MV23}.
Let $X_{\infty}/X$ be an unramified $\Z_p$-covering of connected graphs, that is, a tower of connected graphs
\[
X = X_0 \leftarrow X_1 \leftarrow X_2 \leftarrow \cdots
\]
such that $X_n/X$ is an unramified $\Z/p^n \Z$-covering.
Then there are integers $\lambda \geq 0, \mu \geq 0, \nu$ such that
\[
\ord_p(\kappa_{X_n}) = \lambda n + \mu p^n + \nu
\]
holds for $n \gg 0$.

This paper is concerned with the asymptotic behavior for $\Z_p^d$-coverings of connected graphs with $d \geq 1$.
It consists of three main theorems, which are presented in Subsections \ref{ss:thm1}, \ref{ss:thm3}, and \ref{ss:thm2}, respectively.

\subsection{Iwasawa-type asymptotic formula}\label{ss:thm1}

As the first result, we generalize the aforementioned formula in two directions: we allow ramification and consider multiple $\Z_p$-coverings.

\begin{thm}\label{thm:main1}
Let $d \geq 1$ be an integer.
Let $X_{\infty}/X$ be a (possibly ramified) $\Z_p^d$-covering of connected graphs, that is, a tower of connected graphs
\[
X = X_0 \leftarrow X_1 \leftarrow X_2 \leftarrow \cdots
\]
such that $X_n/X$ is a $(\Z/p^n\Z)^d$-covering.
We define non-negative integers $\lambda= \lambda(X_{\infty}/X)$ and $\mu = \mu(X_{\infty}/X)$ by
\[
\lambda = l_0(\Jac_{\Z_p} X_{\infty})
\AND
\mu = m_0(\Jac_{\Z_p} X_{\infty}),
\]
where $\Jac_{\Z_p} X_{\infty}$ denotes the Jacobian group (Definition \ref{defn:inf_PJ}) and $l_0(-)$ and $m_0(-)$ are module-theoretic invariants (Definition \ref{defn:CM_inv_mod}).
Then there exist rational numbers
\[
\lambda_1, \dots, \lambda_{d-1}, \mu_1, \dots, \mu_{d-1}, \nu
\]
such that
\[
\ord_p(\kappa_{X_n})
= (\lambda n + \mu p^n) p^{(d-1)n} + \sum_{i=1}^{d-1} (\lambda_i n + \mu_i p^n) p^{(d-1-i)n} + \nu
\]
holds for $n \gg 0$.
\end{thm}

We now review previous works (see Table \ref{tab:works}).
When $X_{\infty}/X$ is unramified and $d = 1$, Theorem \ref{thm:main1} specializes to the aforementioned result of Gonet \cite[Theorem 1.1]{Gon22} and McGown--Valli\`eres \cite[Theorem 6.1]{MV23}.
When $X_{\infty}/X$ is unramified and $d \geq 1$, the result is obtained independently by DuBose-Valli\`eres \cite[Theorem A]{DV23} and Kleine--M\"uller \cite[Theorem 4.3]{KlM24}.
Note that examples with non-integer coefficients already appear in \cite[Section 7]{DV23}.

\begin{table}[htbp]
\centering
\begin{tabular}{c | c c}
 & unramified & possibly ramified \\ \hline
$d = 1$ & \begin{tabular}{c} Gonet \cite{Gon22} \\ McGown--Valli\`eres \cite{MV23} \end{tabular} & Gambheera--Valli\`eres \cite{GV24}\\[1em]
$d \geq 1$ & \begin{tabular}{c} DuBose--Valli\`eres \cite{DV23} \\ Kleine--M\"uller \cite{KlM24} \end{tabular}& \begin{tabular}{c} Kundu--M\"uller \cite{KuM24_pre} (up to $O(p^{(d-1)n})$) \\ This paper \end{tabular}\\
\end{tabular}
\caption{Known results on Iwasawa-type asymptotic formulas}
\label{tab:works}
\end{table}

When ramification is allowed but $d = 1$, the result is proved by Gambheera--Valli\`eres \cite[Theorem A]{GV24}; in fact, the appropriate formalism of ramified coverings in this context is introduced there.
Finally, in the ramified case with $d \geq 1$, Kundu--M\"uller \cite{KuM24_pre} obtained a rougher asymptotic formula
\[
\ord_p(\kappa_{X_n})
= (\lambda n + \mu p^n) p^{(d-1)n} + O(p^{(d-1)n})
\]
with the same descriptions of $\lambda$ and $\mu$.
In the present paper, we refine the error term by showing that it is in fact given by a polynomial of $p^n$ and $n$.
This resolves one of the questions raised in the ``Future Directions'' section of \cite{KuM24_pre}.

The proof of Theorem \ref{thm:main1} crucially depends on the work of Monsky \cite{Mon81a} and Cuoco--Monsky \cite{CM81} on multi-variable $p$-adic power series (see Subsection \ref{ss:Mon}).
The theory was first applied to classical Iwasawa theory in \cite[Theorem I]{CM81}; for a $\Z_p^d$-extension of number fields, with $K_n$ denoting its $n$-th layer, they proved the asymptotic formula
\[
\ord_p(h_{K_n})
= (\lambda n + \mu p^n) p^{(d-1)n} + O(p^{(d-1)n}).
\]
As explained in \cite[Section 7]{CM81}, Greenberg raised the question of whether we can obtain the more refined polynomial-type formula without the error term.
While negative evidence is provided in \cite{CM81}, the question remains open.
On the other hand, Wan \cite{Wan19} proved that the polynomial-type formula holds for $\Z_p^d$-extensions of function fields.
From this perspective, the graph-theoretic analogue appears to be rather subtle, and Theorem \ref{thm:main1} provides an affirmative answer.

\subsection{Kida-type formula}\label{ss:thm3}

In classical Iwasawa theory, Kida's formula \cite{Kid80} describes the behavior of the $\lambda$-invariants for $p$-extensions between $\Z_p$-extensions, assuming $\mu = 0$ (more precisely, it concerns the minus components for CM-fields).
Subsequently, analogous formulas are discovered in other areas of Iwasawa theory; for example, for Selmer groups of elliptic curves by Hachimori--Matsuno \cite{HM99}.

We now consider a $p$-covering between $\Z_p^d$-coverings of connected graphs.
To be precise, let $X_{\infty}/X$ and $\wtil{X}_{\infty}/\wtil{X}$ be $\Z_p^d$-coverings of connected graphs with $X_n$ and $\wtil{X}_n$ their $n$-th layers.
We further assume that each $\wtil{X}_n$ is a $G$-covering of $X_n$, where $G$ is a fixed finite $p$-group, and that the actions of $G$ on $\wtil{X}_n$ are compatible as $n$ varies.
The situation is illustrated in the following diagram:
\[
\xymatrix{
\wtil{X}_0 \ar[d]_G & \wtil{X}_1 \ar[d]_G \ar[l] & \wtil{X}_2 \ar[d]_G \ar[l] & \cdots \ar[l]\\
X_0 & X_1 \ar[l] & X_2 \ar[l] & \cdots. \ar[l]
}
\]
The following is the second main theorem of this paper.

\begin{thm}\label{thm:main3}
In this situation, the following hold.
\begin{itemize}
\item[(1)]
We have $\mu(\wtil{X}_{\infty}/\wtil{X}) = 0$ if and only if both $\mu(X_{\infty}/X) = 0$ and $(\star)$ hold:
\begin{itemize}
\item[$(\star)$]
any vertex of $X$ unramified in $X_{\infty}/X$ remains unramified in $\wtil{X}_{\infty}/X$.
\end{itemize}

\item[(2)]
If the equivalent conditions in (1) hold, then we have
\begin{align*}
\lambda(\wtil{X}_{\infty}/\wtil{X}) + \delta_{d, 1}
& = [\wtil{X}_{\infty}: X_{\infty}] (\lambda(X_{\infty}/X) + \delta_{d, 1}) \\
& \quad - \sum_{v \in V_X} l_0(\Z_p[[\wtil{\Gamma}/\wtil{I}_v]]) (m_v(\wtil{X}_{\infty}/X_{\infty}) - 1),
\end{align*}
where
\begin{itemize}
\item
we set $\delta_{d, 1} = 1$ if $d = 1$ and $\delta_{d, 1} = 0$ if $d \geq 2$,
\item
$[\wtil{X}_{\infty}: X_{\infty}]$ denotes the degree of $\wtil{X}_{\infty}/X_{\infty}$, understood to be $\# G$,
\item
$v$ in the sum runs over the vertices of $X$,
\item
$\wtil{I}_v$ denotes the inertia group in the Galois group $\wtil{\Gamma}$ of $\wtil{X}_{\infty}/X$, 
\item
$l_0(-)$ is the module-theoretic invariant over $\Z_p[[\Gamma]]$, with $\Gamma$ the Galois group of $\wtil{X}_{\infty}/\wtil{X}$, and
\item
$m_v(\wtil{X}_{\infty}/X_{\infty})$ denotes the ramification index in $\wtil{X}_{\infty}/X_{\infty}$, understood to be the limit of the ramification indices $m_v(\wtil{X}_n/X_n)$.
\end{itemize}
\end{itemize}
\end{thm}

We have a concrete description of $l_0(\Z_p[[\wtil{\Gamma}/\wtil{I}_v]])$ (see Remark \ref{rem:l0_easy3}).
Note that the $\Z_p[[\Gamma]]$-module $\Z_p[[\wtil{\Gamma}/\wtil{I}_v]]$ is not torsion if and only if $\wtil{I}_v$ is finite.
In this case, although $l_0(\Z_p[[\wtil{\Gamma}/\wtil{I}_v]])$ is not defined, condition $(\star)$ implies $m_v(\wtil{X}_{\infty}/X_{\infty}) = 1$, so we set
\[
l_0(\Z_p[[\wtil{\Gamma}/\wtil{I}_v]]) (m_v(\wtil{X}_{\infty}/X_{\infty}) - 1) = 0.
\]

We now review previous works (see Table \ref{tab:works_Kida}).
The first result on Kida's formula for graphs is obtained in the case of unramified $\Z_p$-coverings by Ray--Valli\`eres \cite[Theorem A]{RV25}.
Then Adachi--Mizuno--Tateno \cite[Theorem 1.2]{AMT26} generalized it to unramified $\Z_p^d$-coverings with $d \geq 1$ (and simultaneously to weighted graphs).
In these unramified cases, the formula becomes much simpler:
\[
\lambda(\wtil{X}_{\infty}/\wtil{X}) + \delta_{d, 1}
= [\wtil{X}_{\infty}: X_{\infty}] (\lambda(X_{\infty}/X) + \delta_{d, 1}).
\]
The author \cite[Theorem 1.1]{Kata_31} proved the formula for possibly ramified $\Z_p$-coverings; in particular, that work clarified the appearance of the condition $(\star)$ and the form of the ramification terms.
This paper establishes the general case, resolving another question raised in the ``Future Directions'' section of \cite{KuM24_pre}.

\begin{table}[htbp]
\centering
\renewcommand{\arraystretch}{1.3}
\begin{tabular}{c | c c}
 & unramified & possibly ramified \\ \hline
$d = 1$ & Ray--Valli\`eres \cite{RV25} & The author \cite{Kata_31}\\
$d \geq 1$ & Adachi--Mizuno--Tateno \cite{AMT26} & This paper \\
\end{tabular}
\caption{Known results on Kida-type formulas}
\label{tab:works_Kida}
\end{table}

\subsection{Possible values of $\lambda$- and $\mu$-invariants}\label{ss:thm2}

As the final topic, we study the possible values of the pair $(\lambda(X_{\infty}/X), \mu(X_{\infty}/X))$ when $X_{\infty}/X$ varies over $\Z_p^d$-coverings of connected graphs, with $p$ and $d$ fixed.

When $d = 1$, this problem was resolved by the author \cite{Kata_36}.
It is shown that $\lambda(X_{\infty}/X)$ is odd for any unramified $\Z_p$-covering of connected graphs and, conversely, that any pair $(\lambda, \mu)$ with $\lambda$ odd can be realized over a bouquet $X$ (i.e., a graph with a single vertex).
When $\lambda$ is even, such a pair $(\lambda, \mu)$ can instead be realized by a ramified $\Z_p$-covering over a graph $X$ with two vertices.

The final main theorem of this paper shows that, when $d \geq 2$, any pair $(\lambda, \mu)$ can be realized by an unramified $\Z_p^d$-covering of a bouquet, without any parity condition.
More formally:

\begin{thm}\label{thm:main2}
We fix a prime number $p$ and $d \geq 2$.
Then for any pair $(\lambda, \mu)$ of non-negative integers, there is an unramified $\Z_p^d$-covering $X_{\infty}/X$ of connected graphs such that
\[
\lambda(X_{\infty}/X) = \lambda
\AND
\mu(X_{\infty}/X) = \mu.
\]
Indeed, we construct such a covering with $X$ a bouquet.
\end{thm}

When $d \geq 2$, both DuBose--Valli\`eres \cite[Section 7]{DV23} and Kleine--M\"uller \cite[Sections 8 and 9]{KlM24} construct some numerical examples.
However, no systematic construction of this kind seems to be known.

\subsection{Structure of this paper}

Sections \ref{sec:alg_pre} and \ref{sec:prelim} are preliminaries on algebra and graphs.
The three main theorems will be proved in Sections \ref{sec:proof}, \ref{sec:Kida}, and \ref{sec:invariants}, respectively.

\section{Preliminaries on algebra}\label{sec:alg_pre}

The first subsection is devoted to review of the results on multi-variable $p$-adic power series by Monsky \cite{Mon81a} and Cuoco--Monsky \cite[Section 1]{CM81}.
In particular, the $l_0$- and $m_0$-invariants of power series are defined.
In Subsection \ref{ss:cha}, we then review the definition of $l_0$- and $m_0$-invariants of modules, and show basic properties.

Throughout this paper, we fix a prime number $p$.
We fix an algebraic closure $\ol{\Q_p}$ of $\Q_p$.
We write $\ord_p: \ol{\Q_p}^{\times} \to \Q$ for the $p$-adic valuation normalized by $\ord_p(p) = 1$.

\subsection{The $p$-adic power series lemma}\label{ss:Mon}

Let $\Gamma$ be a profinite group that is isomorphic to $\Z_p^d$ with $d \geq 1$.

We set $\Gamma_n = \Gamma/\Gamma^{p^n} \simeq (\Z/p^n\Z)^d$ for each $n \geq 0$.
We define the completed group ring by $\Z_p[[\Gamma]] := \varprojlim_n \Z_p[\Gamma_n]$.
Let $\Z_p[[T_1, \dots, T_d]]$ be the $d$-variable formal power series ring.
Taking a basis $\sigma_1, \dots, \sigma_d$ of $\Gamma$, we have an isomorphism
\[
\Z_p[[\Gamma]] \simeq \Z_p[[T_1, \dots, T_d]]
\]
by sending $\sigma_i$ to $1 + T_i$ for each $i = 1, \dots, d$.
This is called Serre's isomorphism.

Let $\widehat{\Gamma}$ denote the group of finite characters of $\Gamma$ with values in $\ol{\Q_p}^{\times}$.
For each closed subgroup $H \subset \Gamma$, we regard the group $\widehat{\Gamma/H}$ of finite characters of $\Gamma/H$ as a subgroup of $\widehat{\Gamma}$.
In particular, $\widehat{\Gamma_n}$ is the subgroup of $\widehat{\Gamma}$ that is isomorphic to $(\Z/p^n\Z)^d$.

For each $\chi \in \widehat{\Gamma}$, the continuous group homomorphism $\chi: \Gamma \to \Imag(\chi) \subset \ol{\Q_p}^{\times}$ is naturally extended to a ring homomorphism
\[
\Z_p[[\Gamma]] \to \Z_p[\Imag(\chi)] \subset \ol{\Q_p},
\]
which is again denoted by $\chi$.
Then for each $f \in \Z_p[[\Gamma]]$ with $\chi(f) \neq 0$, we have a rational number $\ord_p(\chi(f)) \geq 0$.

Given a subset $S \subset \widehat{\Gamma}$ and an element $f \in \Z_p[[\Gamma]]$, we study the asymptotic behavior of the rational numbers
\[
\sum_{\substack{\chi \in S \cap \widehat{\Gamma_n} \\ \chi(f) \neq 0}} \ord_p(\chi(f))
\]
as $n \to \infty$.
Such a formula was obtained by Monsky \cite{Mon81a} when $S$ is semi-algebraic in the sense of \cite[Definition 3.1]{Mon81a}.
Instead of reviewing the precise definition of semi-algebraic sets, we mention the following properties that we need:
\begin{itemize}
\item
$\widehat{\Gamma/H}$ is semi-algebraic for any closed subgroup $H \subset \Gamma$, and
\item
the class of semi-algebraic sets are closed under finite unions, finite intersections, and complements.
\end{itemize}

\begin{thm}[{Monsky \cite[Theorem 5.6]{Mon81a}}]\label{thm:Mon}
Let $f \in \Z_p[[\Gamma]]$ be an element and $S \subset \widehat{\Gamma}$ a semi-algebraic subset.
Then there are rational numbers
\[
\lambda, \mu, \lambda_1, \mu_1, \dots, \lambda_{d-1}, \mu_{d-1}, \nu
\]
such that
\[
\sum_{\substack{\chi \in S \cap \widehat{\Gamma_n} \\ \chi(f) \neq 0}} \ord_p(\chi(f)) 
= (\lambda n + \mu p^n) p^{(d-1)n} + \sum_{i=1}^{d-1} (\lambda_i n + \mu_i p^n) p^{(d-1-i)n} + \nu
\]
holds for $n \gg 0$.
\end{thm}

Note that Serre's isomorphism also gives rise to an isomorphism
\[
\F_p[[\Gamma]] := \varprojlim_n \F_p[\Gamma_n] 
\simeq \F_p[[T_1, \dots, T_d]].
\]
In particular, both $\Z_p[[\Gamma]]$ and $\F_p[[\Gamma]]$ are noetherian unique factorization domains.

\begin{defn}[{\cite[Definitions 1.1 and 1.2]{CM81}}]
For $f \in \F_p[[\Gamma]] \setminus \{0\}$, we define a non-negative integer $l_0(f)$ as
\[
l_0(f) = \sum_P \ord_P(f)
\]
where $P$ runs over the prime ideals of $\F_p[[\Gamma]]$ of the form $(\gamma - 1)$ with $\gamma \in \Gamma \setminus \Gamma^p$ and $\ord_P(-)$ denotes the exponent of (a generator of) $P$ in the prime decomposition of $f$.

For $f \in \Z_p[[\Gamma]] \setminus \{0\}$, we define non-negative integers $l_0(f)$ and $m_0(f)$ by
\[
p^{-m_0(f)} f \in \Z_p[[\Gamma]] \setminus p \Z_p[[\Gamma]]
\]
and $l_0(f) = l_0(\ol{p^{-m_0(f)} f})$, where $\ol{(-)}$ denotes the reduction to $\F_p[[\Gamma]]$.
\end{defn}

When $d = 1$, the $l_0$- and $m_0$-invariants of $f \in \Z_p[[\Gamma]]$ coincide with the usual $\lambda$- and $\mu$-invariants, respectively.

Using these $l_0$- and $m_0$-invariants, we can determine the leading coefficients in Theorem \ref{thm:Mon}:

\begin{thm}[{Cuoco--Monsky \cite[Theorem 1.7]{CM81}}]\label{thm:Cuo-Mon}
Let $f \in \Z_p[[\Gamma]] \setminus \{0\}$.
In Theorem \ref{thm:Mon}, if $S = \widehat{\Gamma}$, we have
\[
\lambda = l_0(f)
\AND
\mu = m_0(f),
\]
which are in particular non-negative integers.
\end{thm}

\subsection{The $l_0$- and $m_0$-invariants of modules}\label{ss:cha}

To ease the notation, we set $\Lambda = \Z_p[[\Gamma]] \simeq \Z_p[[T_1, \dots, T_d]]$.

\begin{defn}\label{defn:CM_inv_mod}
For a finitely generated torsion $\Lambda$-module $M$, we define its characteristic ideal by
\[
\cha(M) = \prod_P P^{\length_{\Lambda_P}(M_P)},
\]
where $P$ runs over the height-one prime ideals of $\Lambda$, and $\Lambda_P$ and $M_P$ denote the localization of $\Lambda$ and $M$ at $P$, respectively.
Note that such prime ideals are nonzero and principal, so $\cha(M)$ is a nonzero principal ideal.

We then define non-negative integers $l_0(M)$ and $m_0(M)$ as the $l_0$- and $m_0$-invariants of any generator of $\cha(M)$.
\end{defn}

Note that characteristic ideals are usually defined via the structure theorem of $\Lambda$-modules up to pseudo-isomorphism, but our alternative definition is more convenient in this paper.

We will often use the following additivity with respect to exact sequences.

\begin{lem}
Let $0 \to M' \to M \to M'' \to 0$ be an exact sequence of finitely generated torsion $\Lambda$-modules.
Then we have $\cha(M) = \cha(M') \cha(M'')$.
In particular, we have
\[
l_0(M) = l_0(M') + l_0(M'')
\AND
m_0(M) = m_0(M') + m_0(M'').
\]
\end{lem}

\begin{proof}
This follows from the facts that localization is exact and that length is additive with respect to exact sequences.
\end{proof}

The following well-known proposition shows that the characteristic ideal coincides with the Fitting ideal if the projective dimension of $M$ is at most one.

\begin{prop}\label{prop:cha_Fitt}
Let $\varphi$ be an injective endomorphism of a finitely generated free $\Lambda$-module.
Let $M = \Cok(\varphi)$ be its cokernel.
Then we have $\cha(M) = (\det \varphi)$.
In particular, we have $l_0(M) = l_0(\det \varphi)$ and $m_0(M) = m_0(\det \varphi)$.
\end{prop}

\begin{proof}
Let $P$ be any height-one prime ideal of $\Lambda$.
Then $M_P$ is the cokernel of the endomorphism of the finitely generated free $\Lambda_P$-module induced by $\varphi$ via base change.
Since $\Lambda_P$ is a discrete valuation ring, we can apply the elementary divisor theory and obtain
\[
\length_{\Lambda_P}(M_P) = \ord_P(\det \varphi).
\]
Therefore,
\[
\cha(M) = \prod_P P^{\length_{\Lambda_P}(M_P)}
= \prod_P P^{\ord_P(\det \varphi)}
= (\det \varphi)
\]
as claimed.
\end{proof}

For later use, we compute the $l_0$- and $m_0$-invariants of an explicit module.

\begin{lem}\label{lem:l0_easy}
Let $I$ be a non-trivial closed subgroup of $\Gamma \simeq \Z_p^d$.
Then we have $m_0(\Z_p[[\Gamma/I]]) = 0$ and 
\[
l_0(\Z_p[[\Gamma/I]])
= \begin{cases} \# (\Gamma/I)_{\tors} & \text{if $\rank_{\Z_p} I = 1$}, \\ 0 & \text{if $\rank_{\Z_p} I \geq 2$}, \end{cases}
\]
where $(\Gamma/I)_{\tors}$ denotes the torsion subgroup of $\Gamma/I$.
\end{lem}

\begin{proof}
If $\rank_{\Z_p} I \geq 2$, then the module $\Z_p[[\Gamma/I]]$ is pseudo-null as a $\Z_p[[\Gamma]]$-module and we obtain $\cha(\Z_p[[\Gamma/I]]) = (1)$.
Suppose $\rank_{\Z_p} I = 1$ and set $p^m = \# (\Gamma/I)_{\tors}$.
We can take a basis $\sigma_1, \dots, \sigma_d$ of $\Gamma$ such that $\sigma_d^{p^m}$ is a basis of $I$. 
Then $\Z_p[[\Gamma/I]] \simeq \Z_p[[\Gamma]]/(\sigma_d^{p^m} - 1)$, so
\[
\cha(\Z_p[[\Gamma/I]]) = (\sigma_d^{p^m} - 1).
\]
After projection to the ring $\F_p[[\Gamma]]$, we have $\sigma_d^{p^m} - 1 = (\sigma_d - 1)^{p^m}$.
Therefore, we have $m_0(\sigma_d^{p^m} - 1) = 0$ and $l_0(\sigma_d^{p^m} - 1) = p^m$.
This completes the proof.
\end{proof}

\begin{rem}
When $d = 1$, the $l_0$- and $m_0$-invariants coincide with the usual $\lambda$- and $\mu$-invariants.
It is known that we have $\mu(M) = 0$ if and only if $M$ is finitely generated over $\Z_p$, and $\lambda(M)$ coincides with the $\Z_p$-rank of $M$.
These descriptions immediately imply Lemma \ref{lem:l0_easy} in the case $d = 1$.
\end{rem}

\section{Preliminaries on graphs}\label{sec:prelim}

The first two subsections are mainly devoted to a brief review of the basics of graph theory, following the author's article \cite[Sections 2 and 3]{Kata_31}.
In Subsection \ref{ss:Pic_der}, we will show formulas for the $l_0$- and $m_0$-invariants of the Picard groups and Jacobian groups that play a key role in the proofs of the main theorems.

\subsection{Graphs and Galois coverings}\label{ss:pre_graphs}

We begin with a formal definition of graphs.

\begin{defn}[{\cite[Definition 2.1]{Kata_31}}]
A graph $X$ is a tuple $(V_X, \bE_X, \ol{\cdot}, s, t)$, where 
\begin{itemize}
\item
$V_X$ and $\bE_X$ are finite sets (vertices and edges respectively),
\item
$\bE_X \to \bE_X, e \mapsto \ol{e}$ is an involution without fixed points, and 
\item
$s, t: \bE_X \to V_X$ are maps satisfying $s(\ol{e}) = t(e)$ and $t(\ol{e}) = s(e)$ for any $e \in \bE_X$.
\end{itemize}
\end{defn}

For each edge $e \in \bE$, the edge $\ol{e}$ is regarded as the opposite of $e$.
Although each edge $e \in \bE_X$ has a direction, we often identify two edges $e, \ol{e}$, making $X$ an undirected graph.

We have a natural notion of morphisms between graphs (\cite[Definition 2.3]{Kata_31}).
Formally, a morphism $f: Y \to X$ consists of maps $f_V: V_Y \to V_X$ and $f_{\bE}: \bE_Y \to \bE_X$ that are compatible with the maps $\ol{\cdot}, s, t$.

\begin{defn}[{\cite[Definition 2.5]{Kata_31}}]
Let $X$ be a graph and $\Gamma$ a finite group.
A $\Gamma$-covering $Y/X$ consists of a morphism $f: Y \to X$ of graphs and an action of $\Gamma$ on $Y$ over $X$ such that
\begin{itemize}
\item
for any $v \in V_X$, the group $\Gamma$ acts on $f_V^{-1}(v)$ transitively, and
\item
for any $e \in \bE_X$, the group $\Gamma$ acts on $f_{\bE}^{-1}(e)$ transitively and freely.
\end{itemize}
In this case, we say that $Y/X$ is a Galois covering with Galois group $\Gamma = \Gal(Y/X)$.
We define the degree of $Y/X$ by $[Y: X] = \# \Gamma$, which is equal to $\# f_{\bE}^{-1}(e)$ for any $e \in \bE_X$.
\end{defn}

\begin{defn}\label{defn:inertia}
Let $\Gamma$ be a finite group and $Y/X$ a $\Gamma$-covering.
For each $v \in V_X$, choosing a vertex $w$ of $Y$ lying above $v$ (i.e., a vertex that is sent to $v$ by the given morphism $Y \to X$), we define the inertia group $I_v(Y/X) \subset \Gamma$ to be the stabilizer of $w$.
Note that $I_v(Y/X)$ depends on the choice of $w$, so it is defined only up to conjugate.
We then define the ramification index by $m_v(Y/X) = \# I_v(Y/X)$, which is independent of the choice of $w$.
We say $v$ is unramified in $Y/X$ if $m_v(Y/X) = 1$.
\end{defn}

We now review how to construct Galois coverings.

\begin{defn}[{\cite[Definition 2.9]{Kata_31}}]\label{defn:der_gr}
Let $X$ be a graph and $\Gamma$ a finite group.
\begin{itemize}
\item
Let $\alpha: \bE_X \to \Gamma$ be a voltage assignment, which is by definition a map such that
\[
\alpha(\ol{e}) = \alpha(e)^{-1}
\]
holds for any $e \in \bE_X$ (\cite[Definition 2.7]{Kata_31}).
\item
Let $\cI = (I_v)_{v \in V_X}$ be a family of subgroups of $\Gamma$ indexed by the set of vertices of $X$.
\end{itemize}
Then we construct the derived graph $X(\Gamma, \alpha, \cI)$ by setting
\[
V_{X(\Gamma, \alpha, \cI)} = \coprod_{v \in V_X} (\Gamma/I_v \times \{v\}),
\quad
\bE_{X(\Gamma, \alpha, \cI)} = \Gamma \times \bE_X
\]
and
\[
\ol{(\gamma, e)} = (\gamma \alpha(e), \ol{e}),
\quad
s((\gamma, e)) = (\gamma I_{s(e)}, s(e)),
\quad
t((\gamma, e)) = (\gamma \alpha(e) I_{t(e)}, t(e))
\]
for any $(\gamma, e) \in \bE_{X(\Gamma, \alpha, \cI)}$.
\end{defn}

Then we have a natural morphism $f: X(\Gamma, \alpha, \cI) \to X$ (the projections to the second components) and a natural action of $\Gamma$ on $X(\Gamma, \alpha, \cI)$ (on the first components).
These structures make $X(\Gamma, \alpha, \cI)/X$ into a $\Gamma$-covering.
For each $v \in V_X$, we have $I_v(X(\Gamma, \alpha, \cI)/X) = I_v$ by choosing $(1_{\Gamma} I_v, v)$ as a vertex lying above $v$ (here $1_{\Gamma}$ denotes the identity element of $\Gamma$).
Conversely, any $\Gamma$-covering can be constructed in this manner, up to isomorphism (\cite[Proposition 2.11]{Kata_31}).

We now discuss the case where $\Gamma$ is a profinite group.
By definition, a $\Gamma$-covering $X_{\infty}/X$ of graphs consists of a family of graphs
\[
X_{\infty} = (X_U)_U,
\]
where $U$ runs over the open normal subgroups of $\Gamma$, such that $X_U/X$ is a $\Gamma/U$-covering and satisfies a certain compatibility condition (\cite[Definition 2.6]{Kata_31}).
For each $v \in V_X$, we define the inertia group by
\[
I_v(X_{\infty}/X) = \varprojlim_U I_v(X_U/X) \subset \Gamma,
\]
which depends on the choice of a compatible family of vertices lying above $v$.
Finally, as in Definition \ref{defn:der_gr}, given
\begin{itemize}
\item
a voltage assignment $\alpha: \bE_X \to \Gamma$ and
\item
a family $\cI = (I_v)_{v \in V_X}$ of closed subgroups of $\Gamma$,
\end{itemize}
we construct a $\Gamma$-covering of $X$ as
\[
X(\Gamma, \alpha, \cI) = (X(\Gamma/U, \alpha_U, \cI_U))_U
\]
(\cite[Definition 2.10]{Kata_31}).
Here, $\alpha_U$ and $\cI_U$ are naturally induced by $\alpha$ and $\cI$ via the projection map $\Gamma \to \Gamma/U$ (\cite[Definition 2.8]{Kata_31}).
We have $I_v(X(\Gamma, \alpha, \cI)/X) = I_v$ for each $v \in V_X$.
Conversely, any $\Gamma$-covering can be constructed in this manner (\cite[Proposition 2.12]{Kata_31}).

Suppose $\Gamma \simeq \Z_p^d$.
Given a voltage assignment $\alpha$ and a family $\cI = (I_v)_{v \in V_X}$, we write $\alpha_n$ and $\cI_n = (I_{v, n})_{v \in V_X}$ for $\alpha_{\Gamma^{p^n}}$ and $\cI_{\Gamma^{p^n}}$.
We then set $X_n = X(\Gamma_n, \alpha_n, \cI_n)$ and obtain a $\Z_p^d$-covering
\[
X = X_0 \leftarrow X_1 \leftarrow X_2 \leftarrow \cdots.
\]
Any $\Z_p^d$-covering can be constructed in this manner.

\subsection{Picard groups and Jacobian groups}\label{ss:PJ}

\begin{defn}[{\cite[Definition 2.2]{Kata_31}}]
For a graph $X$, we define its divisor group as the free $\Z$-module on the set $V_X$, that is,
\[
\Div X = \bigoplus_{v \in V_X} \Z v.
\]
We then define the Laplacian operator $\cL_X$ as the endomorphism of $\Div X$ such that
\[
\cL_X(v) = \sum_{\substack{e \in \bE_X \\ s(e) = v}} (s(e) - t(e))
\]
for each $v \in V_X$.
The Picard group $\Pic X$ is defined by
\[
\Pic X = \Coker(\cL_X) = \Div X / \Imag(\cL_X).
\]

Suppose that $X$ is connected.
We define the degree map $\deg_X: \Div X \to \Z$ as the $\Z$-homomorphism such that $\deg_X(v) = 1$ for each $v \in V_X$.
Then it is easy to see that $\deg_X \circ \cL_X$ is the zero map.
We now define the Jacobian group by
\[
\Jac X = \Ker(\deg_X)/\Imag(\cL_X).
\]
We also define a $\Z$-homomorphism $\iota_X: \Z \to \Div X$ by $\iota_X(1) = \sum_{v \in V_X} v$.
\end{defn}

By definition, we have an exact sequence
\[
0 \to \Jac X \to \Pic X \to \Z \to 0.
\]
It is also easy to see that the sequence
\[
0 \to \Z \overset{\iota_X}{\to} \Div X \overset{\cL_X}{\to} \Div X \overset{\deg_X}{\to} \Z \to 0
\]
is a complex.
Moreover, this complex is exact except for the $\Div X$ on the right, where the homology group is by definition the Jacobian group $\Jac X$ of $X$.

\begin{thm}[Kirchhoff's matrix-tree theorem]\label{thm:Kirch}
For a connected graph $X$, its Jacobian group $\Jac X$ is a finite abelian group whose order is equal to the complexity $\kappa_X$, which is by definition the number of spanning trees of $X$.
\end{thm}

Let us observe functorialities.

\begin{prop}\label{prop:functo}
Let $\Gamma$ be a finite group.
Let $Y/X$ be a $\Gamma$-covering of connected graphs.
\begin{itemize}
\item[(1)]
We define homomorphisms $f_*, f_{\sfr}: \Div Y \to \Div X$ by
\[
f_*(w) = f_V(w)
\AND
f_{\sfr}(w) = \# I_{f_V(w)}(Y/X) f_V(w)
\]
for $w \in V_Y$.
Then the following diagram is commutative.
\[
\xymatrix{
	0 \ar[r] & \Z \ar[r]^-{\iota_Y} \ar[d]_{\# \Gamma} & \Div Y \ar[r]^{\cL_Y} \ar[d]_{f_{\sfr}} & \Div Y \ar[r]^-{\deg_Y} \ar[d]_{f_*} & \Z \ar[r] \ar[d]_{\id} & 0\\
	0 \ar[r] & \Z \ar[r]_-{\iota_X} & \Div X \ar[r]_{\cL_X} & \Div X \ar[r]_-{\deg_X} & \Z \ar[r] & 0.
}
\]
\item[(2)]
We define homomorphisms $f^*, f^{\sfr}: \Div X \to \Div Y$ by
\[
f^*(v) = \sum_{w \in f_V^{-1}(v)} w
\AND
f^{\sfr}(v) = \# I_v(Y/X) \sum_{w \in f_V^{-1}(v)} w
\]
for $v \in V_X$.
Then the following diagram is commutative.
\[
\xymatrix{
	0 \ar[r] & \Z \ar[r]^-{\iota_Y} & \Div Y \ar[r]^{\cL_Y} & \Div Y \ar[r]^-{\deg_Y} & \Z \ar[r] & 0\\
	0 \ar[r] & \Z \ar[r]_-{\iota_X} \ar[u]_{\id} & \Div X \ar[r]_{\cL_X} \ar[u]_{f^*} & \Div X \ar[r]_-{\deg_X} \ar[u]_{f^{\sfr}} & \Z \ar[r] \ar[u]_{\# \Gamma} & 0.
}
\]
\end{itemize}
Even if $Y$ or $X$ is not connected, the homomorphisms $f_*, f_{\sfr}, f^*, f^{\sfr}$ are still defined, and the middle squares involving $\cL_Y$ and $\cL_X$ remain commutative.
\end{prop}

\begin{proof}
Both can be checked directly, and indeed they may be regarded as the $\Z$-linear duals of each other.
For (1), see \cite[Proposition 3.4]{GV24} or \cite[Proposition 3.1]{Kata_31}.
\end{proof}

\begin{defn}[{\cite[Definition 3.4]{Kata_31}}]\label{defn:inf_PJ}
Let $\Gamma$ be a profinite group and $X_{\infty} = (X_U)_U/X$ a $\Gamma$-covering of graphs.
We define a $\Z_p[[\Gamma]]$-module $\Pic_{\Z_p} X_{\infty}$ as the projective limit
\[
\Pic_{\Z_p} X_{\infty} = \varprojlim_U (\Z_p \otimes_{\Z} \Pic X_U)
\]
with respect to the maps induced by $f_*$ in Proposition \ref{prop:functo}(1).
Assuming that each $X_U$ is connected, we define a $\Z_p[[\Gamma]]$-module $\Jac_{\Z_p} X_{\infty}$ by
\[
\Jac_{\Z_p} X_{\infty} = \varprojlim_U (\Z_p \otimes_{\Z} \Jac X_U)
\]
in the same way
\end{defn}

\subsection{Picard groups of derived graphs}\label{ss:Pic_der}

To study the structure of the Picard group, it is convenient to use the formalism of derived graphs.
Let $X$ be a graph and $\Gamma$ a profinite group.
Let $\alpha: \bE_X \to \Gamma$ a voltage assignment and $\cI = (I_v)_{v \in V_X}$ a family of closed subgroups of $\Gamma$.

\begin{defn}\label{defn:Lap_vol}
We define a $\Z[\Gamma]$-endomorphism $\cL_{\alpha}$ of $\bigoplus_{v \in V_X} \Z[\Gamma] v$ by
\[
\cL_{\alpha}(v) = \sum_{\substack{e \in \bE_X \\ s(e) = v}} (s(e) - \alpha(e) t(e)).
\]
\end{defn}

We can describe the Laplacian operators of the derived graphs by using $\cL_{\alpha}$ as follows.

\begin{lem}\label{lem:Lap_two}
Suppose $\Gamma$ is finite and let $Y = X(\Gamma, \alpha, \cI)$ be the derived graph.
Note that we have a natural $\Z[\Gamma]$-isomorphism
\[
\Div Y \simeq \bigoplus_{v \in V_X} \Z[\Gamma/I_v] v.
\]
Then for each $v \in V_X$, we have
\[
\cL_Y(v) = \nu_{I_v} \ol{\cL_{\alpha}(v)},
\]
where $\ol{\cL_{\alpha}(v)} \in \Div Y$ denotes the reduction of $\cL_{\alpha}(v)$ and $\nu_{I_v}$ denotes the norm element of $I_v$.
\end{lem}

\begin{proof}
By direct computation, we have
\begin{align*}
\cL_Y(v) 
& = \sum_{\substack{(\gamma, e) \in \bE_Y \\ s((\gamma, e)) = (1_{\Gamma} I_v, v)}} (s((\gamma, e)) - t((\gamma, e)))\\
& = \sum_{\substack{e \in \bE_X \\ s(e) = v}} \sum_{\gamma \in I_v} ((\gamma I_v, s(e)) - (\gamma \alpha(e) I_{t(e)}, t(e)))\\
& = \nu_{I_v} \sum_{\substack{e \in \bE_X \\ s(e) = v}} ((1_{\Gamma} I_v, s(e)) - (\alpha(e) I_{t(e)}, t(e)))
 = \nu_{I_v} \ol{\cL_{\alpha}(v)}
\end{align*}
as claimed.
\end{proof}

In the rest of this subsection, we consider the case where $\Gamma \simeq \Z_p^d$ with $d \geq 1$ and the derived graph $X_{\infty} = X(\Gamma, \alpha, \cI)$.
We introduce some more notation.

\begin{defn}\label{defn:V0Vur_vol}
Suppose $\Gamma \simeq \Z_p^d$.
We set
\[
V^{\unr} := \{v \in V_X \mid I_v = \{1_{\Gamma}\}\}.
\]
We then define an endomorphism $\cL_{\alpha, V^{\unr}}$ of $\bigoplus_{v \in V^{\unr}} \Z[\Gamma] v$ as the $V^{\unr}$-component of $\cL_{\alpha}$, that is, the composite map
\[
\bigoplus_{v \in V^{\unr}} \Z[\Gamma] v
\hookrightarrow \bigoplus_{v \in V_X} \Z[\Gamma] v
\overset{\cL_{\alpha}}{\to} \bigoplus_{v \in V_X} \Z[\Gamma] v
\twoheadrightarrow \bigoplus_{v \in V^{\unr}} \Z[\Gamma] v,
\]
where the first and the last maps are the natural inclusion and the projection maps, respectively.
We then define a $\Z_p[[\Gamma]]$-module $\Pic_{\Z_p}^{\unr} X_{\infty}$ as the cokernel of the endomorphism of $\bigoplus_{v \in V^{\unr}} \Z_p[[\Gamma]] v$ induced by $\cL_{\alpha, V^{\unr}}$ via base change.
\end{defn}

\begin{prop}\label{prop:compa_Pic_vol}
Suppose $\Gamma \simeq \Z_p^d$.
Then $\Pic_{\Z_p} X_{\infty}$ and $\Pic^{\unr}_{\Z_p} X_{\infty}$ are both finitely generated torsion $\Z_p[[\Gamma]]$-modules and we have
\[
\cha(\Pic_{\Z_p} X_{\infty})
= \cha(\Pic^{\unr}_{\Z_p} X_{\infty}) \prod_{v \in V_X \setminus V^{\unr}} \cha(\Z_p[[\Gamma/I_v]]).
\]
In particular, we obtain
\begin{align*}
l_0(\Pic_{\Z_p} X_{\infty})
& = l_0(\Pic^{\unr}_{\Z_p} X_{\infty}) + \sum_{v \in V_X \setminus V^{\unr}} l_0(\Z_p[[\Gamma/I_v]])\\
& = l_0(\det \cL_{\alpha, V^{\unr}}) + \sum_{\substack{v \in V_X \\ \rank_{\Z_p} I_v = 1}} \#(\Gamma/I_v)_{\tors}
\end{align*}
and
\[
m_0(\Pic_{\Z_p} X_{\infty}) 
= m_0(\Pic^{\unr}_{\Z_p} X_{\infty})
= m_0(\det \cL_{\alpha, V^{\unr}}).
\]
\end{prop}

\begin{proof}
In essence, we generalize the argument in \cite[Proof of Theorem 4.3]{Kata_31} (where our $\Pic_{\Z_p}^{\unr} X_{\infty}$ is denoted by $\Pic_{\Z_p}' X_{\infty}$).
Since $\Div X_n \simeq \bigoplus_{v \in V_X} \Z[\Gamma_n/I_{v, n}] v$, the definition of the Picard group yields an exact sequence
\[
0 \to \Z_p \to \bigoplus_{v \in V_X} \Z_p[\Gamma_n/I_{v, n}] v \overset{\cL_{X_n}}{\to} \bigoplus_{v \in V_X} \Z_p[\Gamma_n/I_{v, n}] v \to \Z_p \otimes_{\Z} \Pic X_n \to 0.
\]
Note that the component of $\cL_{X_n}$ on $\bigoplus_{v \in V^{\unr}} \Z_p[\Gamma_n]$ coincides with $\cL_{\alpha_n, V^{\unr}}$ by Lemma \ref{lem:Lap_two}.
By the description of the transition homomorphisms in Proposition \ref{prop:functo}(1), taking projective limits yields an exact sequence
\[
0 \to \bigoplus_{v \in V^{\unr}} \Z_p[[\Gamma]] v \overset{\cL_{X_{\infty}}}{\to} \bigoplus_{v \in V_X} \Z_p[[\Gamma/I_v]] v \to \Pic_{\Z_p} X_{\infty} \to 0,
\]
where $\cL_{X_{\infty}}$ denotes the induced homomorphism (see \cite[Proposition 3.5(1)]{Kata_31}).
This shows that $\Pic_{\Z_p} X_{\infty}$ is finitely generated and torsion over $\Z_p[[\Gamma]]$.
Moreover, the composition of $\cL_{X_{\infty}}$ with the projection to the $V^{\unr}$-component coincides with $\cL_{\alpha, V^{\unr}}$.
Then the snake lemma yields an exact sequence
\[
0 \to \bigoplus_{v \in V_X \setminus V^{\unr}} \Z_p[[\Gamma/I_v]] v
\to \Pic_{\Z_p} X_{\infty} 
\to \Pic_{\Z_p}^{\unr} X_{\infty}
\to 0.
\]
This implies the formula for the characteristic ideals.
The rest of the statement follows from Proposition \ref{prop:cha_Fitt} and Lemma \ref{lem:l0_easy}.
\end{proof}

\begin{lem}\label{lem:PJ_compa}
Suppose $\Gamma \simeq \Z_p^d$.
Suppose each $X_n$ is connected.
Then we have
\[
\cha(\Pic_{\Z_p} X_{\infty}) = \cha(\Jac_{\Z_p} X_{\infty}) \cha(\Z_p).
\]
In particular, we obtain
\[
l_0(\Pic_{\Z_p} X_{\infty}) = l_0(\Jac_{\Z_p} X_{\infty}) + \delta_{d, 1}
\]
and $m_0(\Pic_{\Z_p} X_{\infty}) = m_0(\Jac_{\Z_p} X_{\infty})$.
\end{lem}

\begin{proof}
By definition we have an exact sequence
\[
0 \to \Jac_{\Z_p} X_{\infty} \to \Pic_{\Z_p} X_{\infty} \to \Z_p \to 0.
\]
Since $\Z_p[[\Gamma]]$ is noetherian, we deduce from Proposition \ref{prop:compa_Pic_vol} that $\Jac_{\Z_p} X_{\infty}$ is also finitely generated and torsion.
Then this sequence shows the formula for the characteristic ideals.
The rest of the statement follows from Lemma \ref{lem:l0_easy} applied to $I = \Gamma$.
\end{proof}

\section{Iwasawa-type asymptotic formula}\label{sec:proof}

In this section we prove Theorem \ref{thm:main1}.
In fact, we prove the following equivalent formulation in terms of derived graphs, which is more suitable for explicit computations and constructions.
The equivalence follows easily from Proposition \ref{prop:compa_Pic_vol} and Lemma \ref{lem:PJ_compa}.

\begin{thm}\label{thm:main1_r}
Let $X$ be a graph and $\Gamma \simeq \Z_p^d$.
Let $\alpha: \bE_X \to \Gamma$ be a voltage assignment and $\cI = (I_v)_{v \in V_X}$ a family of closed subgroups of $\Gamma$.
Let $X_{\infty}/X$ be the derived $\Z_p^d$-covering and suppose that each $X_n$ is connected.
We define
\[
\lambda = l_0(\det \cL_{\alpha, V^{\unr}}) + \sum_{\substack{v \in V_X \\ \rank_{\Z_p} I_v = 1}} \#(\Gamma/I_v)_{\tors} - \delta_{d, 1}
\]
and $\mu = m_0(\det \cL_{\alpha, V^{\unr}})$.
Then there exist rational numbers
\[
\lambda_1, \dots, \lambda_{d-1}, \mu_1, \dots, \mu_{d-1}, \nu
\]
such that
\[
\ord_p(\kappa_{X_n})
= (\lambda n + \mu p^n) p^{(d-1)n} + \sum_{i=1}^{d-1} (\lambda_i n + \mu_i p^n) p^{(d-1-i)n} + \nu
\]
holds for $n \gg 0$.
\end{thm}

In Subsection \ref{ss:fin_cov}, we show Proposition \ref{prop:det_L_compu} on finite Galois coverings of connected graphs.
We will use it to prove Theorem \ref{thm:main1_r} in Subsection \ref{ss:pf_main1}.
In Subsection \ref{ss:examples_main1}, we present some numerical examples.

\subsection{Formulas for finite coverings}\label{ss:fin_cov}

Let $\Gamma$ be a finite group.
For a $\Gamma$-covering $Y/X$ of connected graphs, the divisor group $\Div Y$ is naturally a $\Z[\Gamma]$-module.
We write $(\Div Y)^{\Gamma}$ for its $\Gamma$-invariant part.
Since the Laplacian operator $\cL_Y$ is a $\Z[\Gamma]$-endomorphism of $\Div Y$, it induces an endomorphism $\cL_Y^{\nt}$ of $\Div Y/(\Div Y)^{\Gamma}$; here ``$\nt$'' stands for ``non-trivial.''
Note that $\Div Y/(\Div Y)^{\Gamma}$ is a finitely generated free $\Z$-module since so is $\Div Y$.
We write $\det \cL_Y^{\nt} \in \Z$ for the determinant.

\begin{prop}\label{prop:Jac_compa}
Let $Y/X$ be a $\Gamma$-covering of connected graphs.
Then we have
\[
\frac{\kappa_Y}{\kappa_X} = | \det \cL_Y^{\nt} | \cdot \frac{\prod_{v \in V_X} \# I_v}{\# \Gamma},
\]
where $I_v = I_v(Y/X) \subset \Gamma$ denotes the inertia group.
\end{prop}

\begin{proof}
We first note that the vertical arrows in Proposition \ref{prop:functo}(2) are all injective and the cokernel is the complex
\[
\xymatrix{
	0 \ar[r] & 0 \ar[r] & \frac{\Div Y}{f^*(\Div X)} \ar[r]^{\ol{\cL_Y}} & \frac{\Div Y}{f^{\sfr}(\Div X)} \ar[r]^-{\ol{\deg_Y}} & \Z/{\# \Gamma} \Z \ar[r] & 0,
}
\]
where $\ol{\cL_Y}$ and $\ol{\deg_Y}$ are the homomorphisms induced by $\cL_Y$ and $\deg_Y$.
Therefore, taking the homology exact sequence, we obtain an exact sequence
\[
0 \to \Ker(\ol{\cL_Y}) \to \Jac X \to \Jac Y \to \Cok(\ol{\cL_Y}) \to \Z/(\# \Gamma) \Z \to 0.
\]
All modules in this sequence are finite, so we obtain
\[
\frac{\kappa_Y}{\kappa_X} 
= \frac{\# \Jac Y}{\# \Jac X}
= \frac{\# \Cok(\ol{\cL_Y})}{\# \Ker(\ol{\cL_Y})} \frac{1}{\# \Gamma},
\]
using Kirchhoff's matrix-tree theorem (Theorem \ref{thm:Kirch}).
We now consider a commutative diagram with exact rows
\[
\xymatrix{
& & \frac{\Div Y}{f^*(\Div X)} \ar@{=}[r] \ar[d]_{\ol{\cL_Y}} & \frac{\Div Y}{(\Div Y)^{\Gamma}} \ar[d]^{\cL_Y^{\nt}} \\
0 \ar[r] & \frac{f^*(\Div X)}{f^{\sfr}(\Div X)} \ar[r] & \frac{\Div Y}{f^{\sfr}(\Div X)} \ar[r] & \frac{\Div Y}{(\Div Y)^{\Gamma}} \ar[r] & 0.
}
\]
Here, we have $f^*(\Div X) = (\Div Y)^{\Gamma}$ by the definition of $f^*$.
Moreover, the definition of $f^{\sfr}$ implies
\[
f^*(\Div X) / f^{\sfr}(\Div X) \simeq \bigoplus_{v \in V_X} \Z/(\# I_v) \Z.
\]
Then by using the snake lemma, we obtain
\[
\frac{\# \Cok(\ol{\cL_Y})}{\# \Ker(\ol{\cL_Y})}
= \frac{\# \Cok(\cL_Y^{\nt})}{\# \Ker(\cL_Y^{\nt})} \cdot \# \left( \frac{f^*(\Div X)}{f^{\sfr}(\Div X)} \right)
= | \det \cL_Y^{\nt}| \cdot \prod_{v \in V_X} \# I_v,
\]
where we actually have $\Ker(\ol{\cL_Y}) = 0$ and $\Ker(\cL_Y^{\nt}) = 0$.
This completes the proof.
\end{proof}

To state Proposition \ref{prop:det_L_compu}, we introduce some notation (cf.~Definition \ref{defn:V0Vur_vol}).
Let $X$ be a graph.
Let $\Gamma$ be a finite abelian group.
Let $\alpha: \bE_X \to \Gamma$ be a voltage assignment and $\cI = (I_v)_{V \in V_X}$ a family of subgroups of $\Gamma$.

\begin{defn}\label{defn:sets_V}
We set
\[
V^{\unr} = \{ v \in V_X \mid I_v = \{1_{\Gamma}\}\}.
\]
For each finite character $\chi$ of $\Gamma$, we also set
\[
V^{\chi} = \{ v \in V_X \mid \Ker(\chi) \supset I_v\}.
\]
Note that we have $V^{\unr} \subset V^{\chi} \subset V_X$ for any finite character $\chi$.
For each subset $V \subset V_X$, we define an endomorphism $\cL_{\alpha, V}$ of $\bigoplus_{v \in V} \Z[\Gamma] v$ as the $V$-component of $\cL_{\alpha}$ in Definition \ref{defn:Lap_vol}.
\end{defn}

\begin{prop}\label{prop:det_L_compu}
Let $Y = X(\Gamma, \alpha, \cI)$ be the derived graph with $\Gamma$ a finite abelian group.
We assume that $Y$ is connected.
Then we have
\[
\ord_p(\kappa_Y)
= \sum_{\chi \in \widehat{\Gamma} \setminus \{\bm{1}\}} \ord_p(\chi(\det \cL_{\alpha, V^{\chi}})) + \sum_{\chi \in \widehat{\Gamma}} \sum_{v \in V^{\chi}} \ord_p(\# I_{v}) - \ord_p(\# \Gamma) + \ord_p(\kappa_{X}),
\]
where $\widehat{\Gamma}$ denotes the group of finite characters of $\Gamma$ with values in $\ol{\Q_p}^{\times}$ and $\bm{1} \in \widehat{\Gamma}$ the trivial character.
\end{prop}

\begin{proof}
By Proposition \ref{prop:Jac_compa}, it is enough to show
\[
\ord_p(\det \cL_Y^{\nt})
= \sum_{\chi \in \widehat{\Gamma} \setminus \{\bm{1}\}} \left[ \ord_p(\chi(\det \cL_{\alpha, V^{\chi}})) + \sum_{v \in V^{\chi}} \ord_p(\# I_v) \right].
\]
Since $\Gamma$ is abelian, the characters in $\widehat{\Gamma} \setminus \{\bm{1}\}$ define an isomorphism
\[
\ol{\Q_p}[\Gamma]/\ol{\Q_p}[\Gamma]^{\Gamma} \simeq \prod_{\chi \in \widehat{\Gamma} \setminus \{\bm{1}\}} \ol{\Q_p}.
\]
This decomposition yields $\det \cL_Y^{\nt} = \prod_{\chi \in \widehat{\Gamma} \setminus \{\bm{1}\}} \chi(\det \cL_Y)$ and so
\[
\ord_p(\det \cL_Y^{\nt})
= \sum_{\chi \in \widehat{\Gamma} \setminus \{\bm{1}\}} \ord_p(\chi(\det \cL_Y)).
\]
Fix $\chi \in \widehat{\Gamma} \setminus \{\bm{1}\}$.
Since $\cL_Y$ is a $\Z[\Gamma]$-endomorphism of $\Div Y \simeq \bigoplus_{v \in V_X} \Z[\Gamma/I_v] v$, the value $\chi(\det \cL_Y)$ is the determinant of the induced $\Z[\Imag(\chi)]$-endomorphism $\cL_Y^{\chi}$ of $\bigoplus_{v \in V^{\chi}} \Z[\Imag(\chi)] v$.
By Lemma \ref{lem:Lap_two}, we have
\[
\cL_Y^{\chi}(v) = (\# I_v) \cL_{\alpha, V^{\chi}}(v)
\]
for each $v \in V^{\chi}$.
We then obtain
\[
\chi(\det \cL_Y)
= \det \cL_Y^{\chi}
= \left( \prod_{v \in V^{\chi}} \# I_v \right) \det \cL_{\alpha, V^{\chi}}.
\]
This completes the proof.
\end{proof}

\subsection{Proof of Theorem \ref{thm:main1_r}}\label{ss:pf_main1}

We are now ready to prove Theorem \ref{thm:main1_r}.

\begin{proof}[Proof of Theorem \ref{thm:main1_r}]
By applying Proposition \ref{prop:det_L_compu} to $X_n = X(\Gamma_n, \alpha_n, \cI_n)/X$, we obtain
\[
\ord_p(\kappa_{X_n})
= \sum_{\chi \in \widehat{\Gamma_n} \setminus \{\bm{1}\}} \ord_p(\chi(\det \cL_{\alpha, V^{\chi}})) + \sum_{\chi \in \widehat{\Gamma_n}} \sum_{v \in V^{\chi}} \ord_p(\# I_{v, n}) - dn + \ord_p(\kappa_{X}).
\]
Here, we note that the set $V^{\chi}$ is independent of $n$.
Let us show that the first and the second terms are both polynomials of $p^n$ and $n$ as claimed, and also determine the leading coefficients up to $O(p^{(d-1)n})$.

We begin with the second term:
\begin{align*}
\sum_{\chi \in \widehat{\Gamma_n}} \sum_{v \in V^{\chi}} \ord_p(\# I_{v, n})
& = \sum_{v \in V_X} \ord_p(\# I_{v, n}) \sum_{\substack{\chi \in \widehat{\Gamma_n} \\ V^{\chi} \ni v}} 1
= \sum_{v \in V_X} \ord_p(\# I_{v, n}) \sum_{\chi \in \widehat{\Gamma_n/I_{v, n}}} 1\\
& = \sum_{v \in V_X} \ord_p(\# I_{v, n}) \#(\Gamma_n/I_{v, n}).
\end{align*}
For each $v \in V_X$, we observe that
\[
\ord_p(\# I_{v, n}) = (\rank_{\Z_p} I_v)  n-\ord_p(\# (\Gamma/I_v)_{\tors})
\]
and
\[
\#(\Gamma_n/I_{v, n}) = \#(\Gamma/I_v)_{\tors} p^{(d-\rank_{\Z_p} I_v)n}
\]
hold for $n \gg 0$.
Therefore, we have
\[
\sum_{\chi \in \widehat{\Gamma_n}} \sum_{v \in V^{\chi}} \ord_p(\# I_{v, n})
= \sum_{v \in V_X} \#(\Gamma/I_v)_{\tors}((\rank_{\Z_p} I_v) n-\ord_p(\# (\Gamma/I_v)_{\tors})) p^{(d-\rank_{\Z_p} I_v)n}
\]
for $n \gg 0$.
This is indeed a polynomial of $p^n$ and $n$ as in the theorem.
Moreover, up to $O(p^{(d-1)n})$, this is
\[
\left(\sum_{\substack{v \in V_X \\ \rank_{\Z_p} I_v = 1}} \#(\Gamma/I_v)_{\tors} \right) n p^{(d-1)n}.
\]

For the first term, a key step is to rewrite the sum as
\[
\sum_{\chi \in \widehat{\Gamma_n} \setminus \{\bm{1}\}} \ord_p(\chi(\det \cL_{\alpha, V^{\chi}}))
= \sum_{V^{\unr} \subset V \subset V_X} \left[ \sum_{\substack{\chi \in \widehat{\Gamma_n} \setminus \{ \bm{1} \} \\ V^{\chi} = V}} \ord_p(\chi(\det \cL_{\alpha, V})) \right].
\]
Fix $V^{\unr} \subset V \subset V_X$.
Then 
\begin{align*}
\{\chi \in \widehat{\Gamma} \mid V^{\chi} = V\}
& = \bigcap_{v \in V \setminus V^{\unr}} \{\chi \in \widehat{\Gamma} \mid \Ker(\chi) \supset I_v\} 
	\cap \bigcap_{v \in V_X \setminus V} \{\chi \in \widehat{\Gamma} \mid \Ker(\chi) \not\supset I_v\}\\
& = \bigcap_{v \in V \setminus V^{\unr}} \widehat{\Gamma/I_v} \setminus \bigcup_{v \in V_X \setminus V} \widehat{\Gamma/I_v}.
\end{align*}
Here, the final intersection is taken in the set $\widehat{\Gamma}$, so if $V = V^{\unr}$ we have
\[
\{\chi \in \widehat{\Gamma} \mid V^{\chi} = V^{\unr}\}
= \widehat{\Gamma} \setminus \bigcup_{v \in V_X \setminus V^{\unr}} \widehat{\Gamma/I_v}.
\]
In any case, we see that the set $\{\chi \in \widehat{\Gamma} \setminus \{ \bm{1} \} \mid V^{\chi} = V\}$ is semi-algebraic, so Monsky's theorem (Theorem \ref{thm:Mon}) is applicable.
As a result, the term is a polynomial of $p^n$ and $n$ as desired.
We finally consider the leading coefficients up to $O(p^{(d-1)n})$.
By Theorem \ref{thm:Mon}, for any non-trivial closed subgroup $H$ of $\Gamma$, the contribution of $\widehat{\Gamma/H}$ is absorbed in $O(p^{(d-1)n})$.
Therefore, the leading coefficients come from the term $V = V^{\unr}$, which by Theorem \ref{thm:Cuo-Mon} yields
\[
(l_0(\det \cL_{\alpha, V^{\unr}}) n + m_0(\det \cL_{\alpha, V^{\unr}}) p^n)p^{(d-1)n} + O(p^{(d-1)n}).
\]
This completes the proof of Theorem \ref{thm:main1_r}.
\end{proof}

\begin{rem}\label{rem:inert_same}
Let us discuss a special case where $I_v$ are all the same group $I$ and moreover that $\Gamma/I$ is torsion-free.
In this case, Theorem \ref{thm:main1_r} can be deduced from the unramified case, proved by DuBose--Valli\`eres \cite{DV23}, as follows.
Let $r$ be the rank of $I$, so $\Gamma/I \simeq \Z_p^{d-r}$.
Then $X_{\infty}/X$ contains an unramified $\Gamma/I$-covering $X_{\infty}'/X$.
The situation is illustrated below (although the upper sequence should be drawn diagonally, we use a horizontal layout to save space).
\[
\xymatrix{
	& X_1 \ar[d] \ar[ld]
 	& X_2 \ar[l] \ar[d]
	& \cdots \ar[l]\\
	X
	& X_1' \ar[l]
 	& X_2' \ar[l]
	& \cdots \ar[l]
}
\]
Then the unramified case tells us a formula of the form
\[
\ord_p(\kappa_{X_n'})
= (\lambda' n + \mu' p^n) p^{(d-r-1)n} + \sum_{i=1}^{d-r-1} (\lambda_i' n + \mu_i' p^n) p^{(d-r-1-i)n} + \nu'
\]
for $n \gg 0$.
Note that this holds even if $r = d$ (i.e., $I = \Gamma$), where $X_n' \simeq X$ and thus $\ord_p(\kappa_{X_n'}) = \ord_p(\kappa_X)$ is constant.
Since $X_n/X_n'$ is a $(\Z/p^n\Z)^r$-covering that is totally ramified at all vertices, we have
\[
\kappa_{X_n} 
= (p^{rn})^{\# V_{X_n'}-1} \kappa_{X_n'}
= p^{rn (p^{(d-r)n} \# V_X-1)} \kappa_{X_n'}
\]
(this follows from counting the spanning trees or applying Kirchhoff's matrix-tree theorem).
Therefore, we obtain
\begin{align*}
\ord_p(\kappa_{X_n})
& = rn (p^{(d-r)n} \# V_X-1) + \ord_p(\kappa_{X_n'})\\
& = r \# V_X n p^{(d-r)n} + (\lambda' n + \mu' p^n) p^{(d-r-1)n} + \sum_{i=1}^{d-r-1} (\lambda_i' n + \mu_i' p^n) p^{(d-r-1-i)n} - r n + \nu'.
\end{align*}
This is indeed consistent with our results.
Note that $(\lambda, \mu) = (0, 0)$ if $r \geq 2$ and $(\lambda, \mu) = (\# V_X, 0)$ if $r = 1$.

More generally, this observation is applicable when $I_v$'s are commensurable to each other; it is enough to consider $X_{\infty}/X_{n_0}$ for some large $n_0 \geq 0$.
\end{rem}

\subsection{Numerical examples}\label{ss:examples_main1}

We present numerical results obtained by using SageMath.
Consider the case $\Gamma \simeq \Z_p^2$ with basis $\sigma, \tau$.
Then the asymptotic formula takes the form
\[
\ord_p(\kappa_{X_n}) = \lambda np^n + \mu p^{2n} + \lambda_1n + \mu_1p^n + \nu.
\]
In each case below, the computation are carried out for $n = 0, 1, 2, \dots, 6$, from which we obtain suggested formulas.
Although $\lambda$ and $\mu$ are theoretically determined, identifying $\lambda_1$, $\mu_1$, and $\nu$ would require a more detailed analysis, which we do not pursue in this paper.

Firstly, we consider the case where $X$ is a bouquet.
In this case, Remark \ref{rem:inert_same} is always applicable, so we are essentially reduced to the case of unramified coverings.
Let us demonstrate this observation in an example:

\begin{eg}\label{eg:bou_ram}
Let $p = 2$.
Let $X$ be a bouquet that consists of a single vertex $v$ and $3$ loops $e_1, e_2, e_3$ (together with their opposites). 
We define the voltage assignment $\alpha$ by
\[
\alpha(e_1) = \sigma^2,
\quad
\alpha(e_2) = \alpha(e_3) = \sigma.
\]
Then the derived unramified $\lrang{\sigma} \simeq \Z_p$-covering $X'_{\infty}/X$ is the one that is constructed in \cite{Kata_36} so that we have
\[
\ord_p(\kappa_{X_n'}) = \lambda' n + \mu' p^n + \nu'
\]
with $(\lambda', \mu') = (3, 0)$.

We take the inertia group $I_v = \lrang{\tau} \simeq \{0\} \times \Z_p$.
The resulting coverings are illustrated as follows, where the lower tower is the unramified $\lrang{\sigma}$-covering and the upper tower is the $\Gamma$-covering.
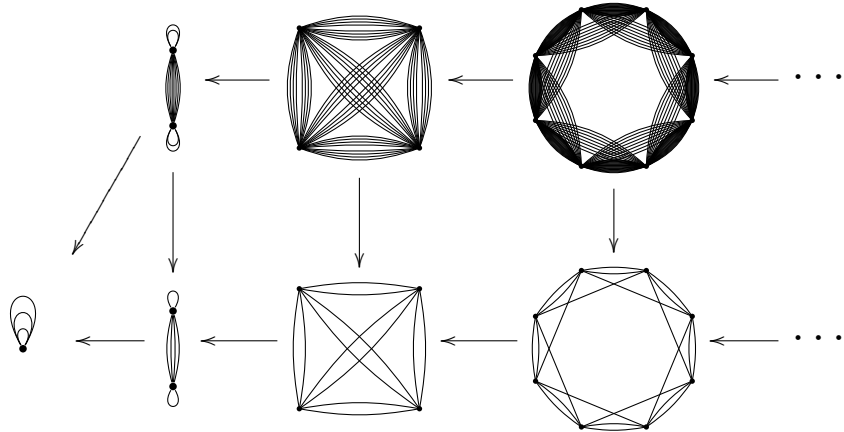
\[
\Large
\xymatrix{
&
\begin{tikzpicture}[baseline]
\node[circle,fill,inner sep=1pt] (a) at (0,0.5) {};
\node[circle,fill,inner sep=1pt] (b) at (0,-0.5) {};
\foreach \bend in {17.5, 12.5, 7.5, 2.5, -2.5, -7.5, -12.5, -17.5}
\path (a) edge[draw,bend right=\bend] (b);
\path (a) edge[loop left, looseness=15, out=120, in=60] (a);
\path (a) edge[loop left, looseness=20, out=120, in=60] (a);
\path (b) edge[loop left, looseness=15, out=-120, in=-60] (b);
\path (b) edge[loop left, looseness=20, out=-120, in=-60] (b);
\end{tikzpicture}
\ar[ld] \ar[d]
&
\begin{tikzpicture}[baseline]
	\node[minimum size=4em,regular polygon,regular polygon sides=4] (a) {};
	\foreach \x in {1, 2, ..., 4}
	\fill (a.corner \x) circle[radius=1pt];
	\foreach \x/\y in {1/2, 2/3, 3/4, 4/1}
	\foreach \bend in {20,15,10,5,-5,-10,-15,-20}
		\path (a.corner \x) edge [bend right=\bend] (a.corner \y);
	\foreach \x/\y in {1/3, 2/4, 3/1, 4/2}
	\foreach \bend in {5,10,15,20}
		\path (a.corner \x) edge [bend right=\bend] (a.corner \y);
\end{tikzpicture}
\ar[l] \ar[d]
&
\begin{tikzpicture}[baseline]
	\node[minimum size=4em,regular polygon,regular polygon sides=8] (a) {};
	\foreach \x in {1, 2, ..., 8}
	\fill (a.corner \x) circle[radius=1pt];
	\foreach \x/\y in {1/2, 2/3, 3/4, 4/5, 5/6, 6/7, 7/8, 8/1}
	\foreach \bend in {20,17.5,15,12.5,10,7.5,5,2.5,0, -2.5,-5,-7.5,-10,-12.5,-15,-17.5,-20}
		\path (a.corner \x) edge [bend right=\bend] (a.corner \y);
	\foreach \x/\y in {1/3, 2/4, 3/5, 4/6, 5/7, 6/8, 7/1, 8/2}
	\foreach \bend in {15,10,5,0,-5,-10,-15,-20}
		\path (a.corner \x) edge [bend right=\bend] (a.corner \y);
\end{tikzpicture}
\ar[l] \ar[d]
& \cdots \ar[l]
\\
\begin{tikzpicture}[baseline]
\node[circle,fill,inner sep=1pt] (a) {};
\foreach \len in {15,30}
\path (a) edge[loop above, looseness=\len, out=60, in=120] (a);
\path (a) edge[loop above, looseness=45, out=60, in=120] node[above] {} (a);
\end{tikzpicture}
&
\begin{tikzpicture}[baseline]
\node[circle,fill,inner sep=1pt] (a) at (0,0.5) {};
\node[circle,fill,inner sep=1pt] (b) at (0,-0.5) {};
\foreach \bend in {15, 5, -5,-15}
\path (a) edge[draw,bend right=\bend] (b);
\path (a) edge[loop left, looseness=15, out=120, in=60] (a);
\path (b) edge[loop left, looseness=15, out=-120, in=-60] (b);
\end{tikzpicture}
\ar[l] &
\begin{tikzpicture}[baseline]
	\node[minimum size=4em,regular polygon,regular polygon sides=4] (a) {};
	\foreach \x in {1, 2, ..., 4}
	\fill (a.corner \x) circle[radius=1pt];
	\foreach \x/\y in {1/2, 2/3, 3/4, 4/1}
	\foreach \bend in {10, -10}
		\path (a.corner \x) edge [bend right=\bend] (a.corner \y);
	\foreach \x/\y in {1/3, 2/4, 3/1, 4/2}
	\foreach \bend in {10}
		\path (a.corner \x) edge [bend right=\bend] (a.corner \y);
\end{tikzpicture}
\ar[l] &
\begin{tikzpicture}[baseline]
	\node[minimum size=4em,regular polygon,regular polygon sides=8] (a) {};
	\foreach \x in {1, 2, ..., 8}
	\fill (a.corner \x) circle[radius=1pt];
	\foreach \x/\y in {1/2, 2/3, 3/4, 4/5, 5/6, 6/7, 7/8, 8/1}
	\foreach \bend in {10, -10}
		\path (a.corner \x) edge [bend right=\bend] (a.corner \y);
	\foreach \x/\y in {1/3, 2/4, 3/5, 4/6, 5/7, 6/8, 7/1, 8/2}
	\foreach \bend in {0}
		\path (a.corner \x) edge [bend right=\bend] (a.corner \y);
\end{tikzpicture}
\ar[l] & \cdots \ar[l]
}
\]
Then by Remark \ref{rem:inert_same}, we obtain an asymptotic formula
\[
\ord_p(\kappa_{X_n}) 
= n(p^n-1) + \lambda' n + \mu' p^n + \nu'
= np^n + (\lambda'-1) n + \mu' p^n + \nu'.
\]
The numerical results are listed in Table \ref{tab:eg_dum}.
The computations suggest that the formula is $np^n + 2n + 1$ for $n \geq 2$, which is consistent with our prediction.
\end{eg}

\begin{table}[htbp]
\centering
\begin{tabular}{c | @{\hspace{1em}} c c c c c c c @{\hspace{1em}} c}
\toprule
 & \swid{$0$} & \swid{$1$} & \swid{$2$} & \swid{$3$} & \swid{$4$} & \swid{$5$} & \swid{$6$} &Suggested formula\\ \midrule
Example \ref{eg:bou_ram} ($p = 2$) & 0 & 3 & 13 & 31 & 73 & 171 & 397 & $n p^n + 2n + 1$ ($n \geq 2$)\\
Example \ref{eg:dum1} ($p = 2$) & 0 & 2 & 12 & 42 & 120 & 310 & 756 & $2n p^n - 2n$ ($n \geq 0$)\\
Example \ref{eg:dum1} ($p = 3$) & 0 & 4 & 32 & 156 & 640 & 2420 & 8736 & $2n p^n - 2n$ ($n \geq 0$)\\
Example \ref{eg:dum2} ($p = 2$) & 0 & 2 & 14 & 52 & 154 & 408 & 1014 & $3np^n-2p^n-2n+2$ ($n \geq 0$)\\
Example \ref{eg:dum2} ($p = 3$) & 0 & 3 & 43 & 239 & 1047 & 4123 & 15299 & $4np^n-3p^n-2n+2$ ($n \geq 1$)\\
Example \ref{eg:dum3} ($p = 2$) & 1 & 6 & 24 & 88 & 320 & 1184 & 4480 & $p^{2n} + np^n$ ($n \geq 0$)\\
\bottomrule
\end{tabular}
\caption{The values of $\ord_p(\kappa_{X_n})$}
\label{tab:eg_dum}
\end{table}

In the following three examples, we consider the case where $X$ consists of two vertices.

\begin{eg}\label{eg:dum1}
Let $X$ be a graph such that $V_X = \{v_1, v_2\}$ and
\[
\bE_X = \{e_1, e_2, e_3, \ol{e_1}, \ol{e_2}, \ol{e_3}\},
\]
where $e_1$ and $e_2$ are loops at $v_1$ and $v_2$ respectively, and $e_3$ is an edge from $v_1$ to $v_2$.
We define the voltage assignment $\alpha$ by
\[
\alpha(e_1) = \tau,
\quad
\alpha(e_2) = \sigma,
\quad
\alpha(e_3) = 1_{\Gamma}.
\]
We define the inertia groups by $I_{v_1} = \lrang{\sigma} \simeq \Z_p \times \{0\}$ and $I_{v_2} = \lrang{\tau} \simeq \{0\} \times \Z_p$.
The associated $\Gamma$-covering when $p = 2$ is illustrated as follows.
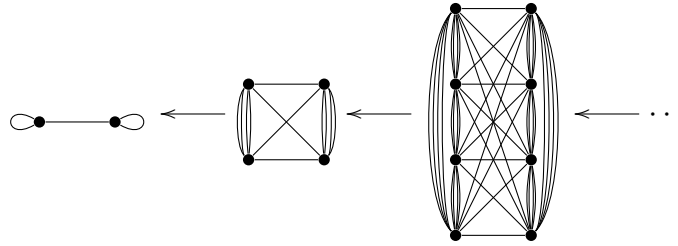
\[
\xymatrix{
\begin{tikzpicture}[baseline={(current bounding box.center)}]
\foreach \x in {0,1}
{
  \node[circle,fill,inner sep=1.5pt]
    (v\x) at (\x,0) {};
}
\draw (v0) -- (v1);
\path (v0) edge[loop, looseness=15, out=150, in=210] (v0);
\path (v1) edge[loop, looseness=15, out=30, in=-30] (v1);
\end{tikzpicture}
&
\begin{tikzpicture}[baseline={(current bounding box.center)}]
\foreach \x in {0,1}
{
  \foreach \y in {0,1}
  {
    \node[circle,fill,inner sep=1.5pt]
      (v\x\y) at (\x,\y) {};
  }
}
\foreach \x in {0,1}
{
  \foreach \y/\yy in {0/1}
    {
      \foreach \bend in {5,-5}
        {
          \path (v\x\y) edge[draw,bend right=\bend] (v\x\yy);
        }
    }
}
\foreach \len in {0.2,0.5}
{
  \draw (v01) to[out=225,in=135,looseness=\len] (v00);
}
\foreach \len in {0.2,0.5}
{
  \draw (v11) to[out=-45,in=45,looseness=\len] (v10);
}
\foreach \y in {0,1}
{
  \foreach \yy in {0,1}
    {
      \draw (v0\y) -- (v1\yy);
    }
}
\end{tikzpicture}
\ar[l]
&
\begin{tikzpicture}[baseline={(current bounding box.center)}]
\foreach \x in {0,1}
{
  \foreach \y in {0,1,2,3}
  {
    \node[circle,fill,inner sep=1.5pt]
      (v\x\y) at (\x,\y) {};
  }
}
\foreach \x in {0,1}
{
  \foreach \y/\yy in {0/1,1/2,2/3}
    {
      \foreach \bend in {10,5,-5,-10}
        {
          \path (v\x\y) edge[draw,bend right=\bend] (v\x\yy);
        }
    }
}
\foreach \len in {0.2,0.3,0.4,0.5}
{
  \draw (v03) to[out=225,in=135,looseness=\len] (v00);
}
\foreach \len in {0.2,0.3,0.4,0.5}
{
  \draw (v13) to[out=-45,in=45,looseness=\len] (v10);
}
\foreach \y in {0,1,2,3}
{
  \foreach \yy in {0,1,2,3}
    {
      \draw (v0\y) -- (v1\yy);
    }
}
\end{tikzpicture}
\ar[l] 
& \cdots \ar[l]}
\]
In this case, we have $V^{\unr} = \emptyset$ and so $\det \cL_{\alpha, V^{\unr}} = 1$.
Then Theorem \ref{thm:main1_r} predicts that $(\lambda, \mu) = (2, 0)$.
The numerical results when $p = 2, 3$ are listed in Table \ref{tab:eg_dum}, which are consistent with the prediction.
\end{eg}

\begin{eg}\label{eg:dum2}
Let the graph $X$, the voltage assignment $\alpha$, and the inertia group $I_{v_2}$ be the same as in Example \ref{eg:dum1}.
We change the inertia group $I_{v_1}$ to $I_{v_1} = \lrang{\sigma^p} \simeq p\Z_p \times \{0\}$.
In this case, Theorem \ref{thm:main1_r} predicts that $(\lambda, \mu) = (p + 1, 0)$.
The numerical results when $p = 2, 3$ are listed in Table \ref{tab:eg_dum}.
\end{eg}

\begin{eg}\label{eg:dum3}
Let us construct an example where $\lambda > 0$ and $\mu > 0$.
Let $X$ be the graph such that $V_X = \{v_1, v_2\}$ and 
\[
\bE_X = \{e_0, e_1, \dots, e_p, \ol{e_0}, \ol{e_1}, \dots, \ol{e_p}\},
\]
where $e_0$ is a loop at $v_1$ and $e_1, \dots, e_p$ are all edges from $v_1$ to $v_2$.
We define the voltage assignment $\alpha$ by
\[
\alpha(e_0) = \tau,
\quad
\alpha(e_1) = \dots = \alpha(e_p) = 1_{\Gamma}.
\]
We define the inertia groups by $I_{v_1} = \lrang{\sigma} \simeq \Z_p \times \{0\}$ and $I_{v_2} = \{1_{\Gamma}\}$.
Then we have $V^{\unr} = \{v_2\}$ and $\cL_{\alpha, V^{\unr}} = \begin{pmatrix} p \end{pmatrix}$.
It follows that $(\lambda, \mu) = (1, 1)$.
The numerical results when $p = 2$ are listed in Table \ref{tab:eg_dum}.
\end{eg}

\section{Kida-type formula}\label{sec:Kida}

In this section, we prove Theorem \ref{thm:main3} by generalizing the strategy of \cite{Kata_31}, where $d = 1$.
In Subsection \ref{ss:Kida_alg}, we prove a key algebraic proposition (Proposition \ref{prop:Kida3}).
We will apply it to the Picard groups and deduce Theorem \ref{thm:main3} in Subsection \ref{ss:Kida_pf}.

\subsection{Algebraic propositions}\label{ss:Kida_alg}

Let $p$ be a prime number and let $\Gamma \simeq \Z_p^d$.
Let $G$ be a finite $p$-group.
Let
\[
\aug_G: \Z_p[[\Gamma]][G] \to \Z_p[[\Gamma]]
\AND
N_G: \Z_p[[\Gamma]][G] \to \Z_p[[\Gamma]]
\]
be the augmentation map and the norm map, respectively.
By definition, for each $F \in \Z_p[[\Gamma]][G]$, its norm $N_G(F)$ is the determinant of the $\Z_p[[\Gamma]]$-endomorphism $\times F$ of $\Z_p[[\Gamma]][G]$.
We let $\aug_G$ and $N_G$ denote the maps for $\F_p$ instead of $\Z_p$ as well.

\begin{lem}\label{lem:Kida1}
Let $F \in \F_p[[\Gamma]][G]$ be an element.
Then we have $N_G(F) = \aug_G(F)^{\# G}$ in $\F_p[[\Gamma]]$.
\end{lem}

\begin{proof}
Let $p^N$ be the exponent of $G$, that is, the maximum of the orders of the elements of $G$.
Since the $p$-th Frobenius map is a ring homomorphism, we have $F^{p^N} = \aug_G(F)^{p^N}$.
Then applying $N_G$, we obtain
\[
N_G(F)^{p^N} 
= N_G(F^{p^N})
= N_G(\aug_G(F)^{p^N})
= (\aug_G(F)^{\# G})^{p^N}.
\]
Since $\F_p[[\Gamma]]$ is a domain of characteristic $p$, this shows $N_G(F) = \aug_G(F)^{\# G}$.
\end{proof}

\begin{prop}\label{prop:Kida2}
Let $F \in \Z_p[[\Gamma]][G]$ be an element.
\begin{itemize}
\item[(1)]
We have $m_0(\aug_G(F)) = 0$ if and only if $m_0(N_G(F)) = 0$.
\item[(2)]
If the equivalent conditions in (1) hold, then we have $l_0(N_G(F)) = (\# G) l_0(\aug_G(F))$.
\end{itemize}
\end{prop}

\begin{proof}
We have commutative diagrams
\[
\xymatrix{
\Z_p[[\Gamma]][G] \ar[r]^-{\aug_G} \ar[d] &
\Z_p[[\Gamma]] \ar[d]
\\
\F_p[[\Gamma]][G] \ar[r]_-{\aug_G} &
\F_p[[\Gamma]]
}
\AND
\xymatrix{
\Z_p[[\Gamma]][G] \ar[r]^-{N_G} \ar[d] &
\Z_p[[\Gamma]] \ar[d]
\\
\F_p[[\Gamma]][G] \ar[r]_-{N_G} &
\F_p[[\Gamma]],
}
\]
where the vertical arrows are the reduction maps, denoted by $\ol{(-)}$.
Then by Lemma \ref{lem:Kida1}, we obtain
\[
\ol{N_G(F)} = \ol{\aug_G(F)}^{\# G}.
\]
In particular, $\ol{N_G(F)} \neq 0$ if and only if $\ol{\aug_G(F)} \neq 0$, so claim (1) follows.
If these equivalent conditions hold, then $l_0(\ol{N_G(F)}) = (\# G) l_0(\ol{\aug_G(F)})$ by the definition of $l_0(-)$, so claim (2) follows.
\end{proof}

The following is the key algebraic proposition, which is a generalization of \cite[Proposition 4.2]{Kata_31}.
For a finitely generated torsion $\Z_p[[\Gamma]][G]$-module $M$, we define $l_0(M)$ and $m_0(M)$ by forgetting the action of $G$.

\begin{prop}\label{prop:Kida3}
Let $M$ be a finitely generated torsion $\Z_p[[\Gamma]][G]$-module whose projective dimension is at most one.
Let $M_G$ be its $G$-coinvariant module, which is a $\Z_p[[\Gamma]]$-module.
\begin{itemize}
\item[(1)]
We have $m_0(M) = 0$ if and only if $m_0(M_G) = 0$.
\item[(2)]
If the equivalent conditions in (1) hold, then we have $l_0(M) = (\# G) l_0(M_G)$.
\end{itemize}
\end{prop}

\begin{proof}
Since $G$ is a $p$-group, by taking a central series and arguing inductively, we may assume that $G$ is abelian.

By assumption, the $\Z_p[[\Gamma]][G]$-module $M$ is the cokernel of an endomorphism $\Phi$ of a free module $\Z_p[[\Gamma]][G]^r$.
Let $F = \det \nolimits_{\Z_p[[\Gamma]][G]}(\Phi) \in \Z_p[[\Gamma]][G]$ be its determinant (here we require $G$ to be abelian).
Let us show the following claims:
\begin{itemize}
\item[(a)]
$\cha(M_G) = (\aug_G(F))$ and
\item[(b)]
$\cha(M) = (N_G(F))$, where $\cha(M)$ is defined by forgetting the action of $G$.
\end{itemize}
Then the proposition follows from Proposition \ref{prop:Kida2}.

We show (a).
Let $\aug_G(\Phi)$ be the endomorphism of the free module $\Z_p[[\Gamma]]^r$ induced by $\Phi$ by taking the $G$-coinvariants.
Then $M_G$ is the cokernel of $\aug_G(\Phi)$, so by Proposition \ref{prop:cha_Fitt}, we have $\cha(M_G) = (\det \nolimits_{\Z_p[[\Gamma]]}(\aug_G(\Phi)))$.
Since $\aug_G: \Z_p[[\Gamma]][G] \to \Z_p[[\Gamma]]$ is a ring homomorphism, we obtain
\[
\det \nolimits_{\Z_p[[\Gamma]]}(\aug_G(\Phi))
= \aug_G (\det \nolimits_{\Z_p[[\Gamma]][G]}(\Phi)) 
= \aug_G(F),
\]
which implies (a).

We show (b).
By Proposition \ref{prop:cha_Fitt}, we have $\cha(M) = (\det \nolimits_{\Z_p[[\Gamma]]}(\Phi))$.
By a standard property of determinants and norms, we have
\[
\det \nolimits_{\Z_p[[\Gamma]]}(\Phi)
= N_G(\det \nolimits_{\Z_p[[\Gamma]][G]}(\Phi))
= N_G(F),
\]
which implies (b).
\end{proof}

\subsection{Proof of Theorem \ref{thm:main3}}\label{ss:Kida_pf}

We are now ready to prove Theorem \ref{thm:main3}.
We first reformulate the theorem in terms of derived graphs.

Let $X$ be a connected graph.
Let $\wtil{\Gamma}$ be a pro-$p$ group of the form $\wtil{\Gamma} = \Gamma \times G$, where $\Gamma \simeq \Z_p^d$ and $G$ is a finite $p$-group.
Let $\wtil{\alpha}: \bE_X \to \wtil{\Gamma}$ be a voltage assignment.
Let $\wtil{\cI} = (\wtil{I_v})_{v \in V_X}$ be a family of closed subgroups of $\wtil{\Gamma}$.
By the projection map $\wtil{\Gamma} \twoheadrightarrow \Gamma$, these induce a voltage assignment $\alpha: \bE_X \to \Gamma$ and a family $\cI = (I_v)_{v \in V_X}$ of closed subgroups of $\Gamma$.
Therefore, we obtain the derived graphs
\[
\wtil{X}_{\infty} = X(\wtil{\Gamma}, \wtil{\alpha}, \wtil{\cI})/\wtil{X}
\AND
X_{\infty} = X(\Gamma, \alpha, \cI)/X.
\]
We assume that the intermediate graphs are all connected.

Then we can reformulate Theorem \ref{thm:main3} as follows.

\begin{thm}
In this situation, the following hold.
\item[(1)]
We have $\mu(\wtil{X}_{\infty}/\wtil{X}) = 0$ if and only if both $\mu(X_{\infty}/X) = 0$ and $(\star)$ hold:
\begin{itemize}
\item[$(\star)$]
for any vertex $v \in V_X$, if $I_v$ is trivial, then $\wtil{I_v}$ is trivial.
\end{itemize}

\item[(2)]
If the equivalent conditions in (1) hold, then we have
\[
\lambda(\wtil{X}_{\infty}/\wtil{X}) + \delta_{d, 1}
 = (\# G) (\lambda(X_{\infty}/X) + \delta_{d, 1}) 
 - \sum_{v \in V_X} l_0(\Z_p[[\wtil{\Gamma}/\wtil{I}_v]]) (\# (G \cap \wtil{I_v}) - 1).
\]
\end{thm}

\begin{proof}
We first show that $\mu(\wtil{X}_{\infty}/\wtil{X}) = 0$ implies $(\star)$.
The same reasoning as \cite[Proof of Theorem 4.3]{Kata_31}, which essentially results from Proposition \ref{prop:functo}(1), gives us an exact sequence
\[
0 \to \bigoplus_{v \in V^0} (\Z_p/(\# \wtil{I}_v) \Z)[[\Gamma]] \to (\Pic_{\Z_p} \wtil{X}_{\infty})_G \to \Pic_{\Z_p} X_{\infty} \to 0,
\]
where we set
\[
V^0 := \{v \in V_X \mid \text{$\wtil{I_v}$ is finite}\}.
\]
Since $(\Pic_{\Z_p} \wtil{X}_{\infty})_G$ is a quotient module of $\Pic_{\Z_p} \wtil{X}_{\infty}$, the assumption $m_0(\Pic_{\Z_p} \wtil{X}_{\infty}) = 0$ implies $m_0((\Pic_{\Z_p} \wtil{X}_{\infty})_G) = 0$.
This then implies
\[
m_0(\Z_p/(\# \wtil{I}_v) \Z)[[\Gamma]]) = 0
\]
for any $v \in V^0$, that is, $\wtil{I}_v$ is trivial for any $v \in V^0$.
Therefore, condition $(\star)$ holds.

In the rest of the proof, we may assume that $(\star)$ holds.
By Theorem \ref{thm:main1} (together with Proposition \ref{prop:compa_Pic_vol} and Lemma \ref{lem:PJ_compa}), we have the formulas
\[
\lambda(X_{\infty}/X) + \delta_{d, 1}
= l_0(\Pic^{\unr}_{\Z_p} X_{\infty}) + \sum_{v \in V_X \setminus V^{\unr}} l_0(\Z_p[[\Gamma/I_v]])
\]
and $\mu(X_{\infty}/X) = m_0(\Pic^{\unr}_{\Z_p} X_{\infty})$.
Here, $V^{\unr} = \{v \in V_X \mid I_v = \{1_{\Gamma}\}\}$ by definition, but condition $(\star)$ implies
\[
V^{\unr} = \{v \in V_X \mid \wtil{I_v} = \{1_{\wtil{\Gamma}}\}\}.
\]
Taking this observation into account, we define a $\Z_p[[\wtil{\Gamma}]]$-module $\Pic_{\Z_p}^{\unr} \wtil{X}_{\infty}$ in the same manner as in Definition \ref{defn:V0Vur_vol}.
To be concrete, we define an endomorphism $\cL_{\wtil{\alpha}, V^{\unr}}$ of $\bigoplus_{v \in V^{\unr}} \Z[\wtil{\Gamma}] v$ as the $V^{\unr}$-component of $\cL_{\wtil{\alpha}}$, and then define $\Pic_{\Z_p}^{\unr} \wtil{X}_{\infty}$ as the cokernel of the endomorphism of $\bigoplus_{v \in V^{\unr}} \Z_p[[\wtil{\Gamma}]] v$ induced by $\cL_{\wtil{\alpha}, V^{\unr}}$.
Then by Theorem \ref{thm:main1} (together with a direct extension of Proposition \ref{prop:compa_Pic_vol} and Lemma \ref{lem:PJ_compa}), we also have
\[
\lambda(\wtil{X}_{\infty}/\wtil{X}) + \delta_{d, 1}
= l_0(\Pic^{\unr}_{\Z_p} \wtil{X}_{\infty}) + \sum_{v \in V_X \setminus V^{\unr}} l_0(\Z_p[[\wtil{\Gamma}/\wtil{I_v}]])
\]
and $\mu(\wtil{X}_{\infty}/\wtil{X}) = m_0(\Pic^{\unr}_{\Z_p} \wtil{X}_{\infty})$.

By definition, $\Pic^{\unr}_{\Z_p} \wtil{X}_{\infty}$ is a finitely generated torsion $\Z_p[[\wtil{\Gamma}]]$-module whose projective dimension is at most one.
Moreover, since the base change of $\cL_{\wtil{\alpha}, V^{\unr}}$ coincides with $\cL_{\alpha, V^{\unr}}$, we have an isomorphism
\[
(\Pic^{\unr}_{\Z_p} \wtil{X}_{\infty})_G \simeq \Pic^{\unr}_{\Z_p} X_{\infty}.
\]
Therefore, the theorem follows from applying Proposition \ref{prop:Kida3} to $\Pic^{\unr}_{\Z_p} \wtil{X}_{\infty}$, if we have
\[
(\# G) l_0(\Z_p[[\Gamma/I_v]]) - l_0(\Z_p[[\wtil{\Gamma}/\wtil{I_v}]])
= l_0(\Z_p[[\wtil{\Gamma}/\wtil{I}_v]]) (\# (G \cap \wtil{I_v}) - 1)
\]
for each $v \in V_X$.
The last equation follows from elementary consideration:
It is enough to show
\[
l_0(\Z_p[[\wtil{\Gamma}/\wtil{I}_v]])
= (G: G \cap \wtil{I_v}) l_0(\Z_p[[\Gamma/I_v]]),
\]
which follows from the observation that the projection map yields a $(G: G \cap \wtil{I_v})$-to-one map
\[
\wtil{\Gamma}/\wtil{I_v} \twoheadrightarrow \Gamma/I_v.
\]
This completes the proof.
\end{proof}

\begin{rem}\label{rem:l0_easy3}
For a closed subgroup $\wtil{I}$ of $\wtil{\Gamma}$, we have an explicit formula
\[
l_0(\Z_p[[\wtil{\Gamma}/\wtil{I}]])
= \begin{cases} (\wtil{I}^{\sat}: \wtil{I}) & \text{if $\rank_{\Z_p} I = 1$}, \\ 0 & \text{if $\rank_{\Z_p} I \geq 2$}. \end{cases}
\]
Here, we define the saturation $G \subset \wtil{I}^{\sat} \subset \wtil{\Gamma}$ of $\wtil{I}$ by $\wtil{I}^{\sat}/\wtil{I}$ being finite and $\wtil{\Gamma}/\wtil{I}^{\sat}$ being torsion-free.
This formula follows from $l_0(\Z_p[[\wtil{\Gamma}/\wtil{I}_v]]) = (G: G \cap \wtil{I_v}) l_0(\Z_p[[\Gamma/I_v]])$ combined with Lemma \ref{lem:l0_easy}; we omit details.
\end{rem}

\section{Possible values of $\lambda$- and $\mu$-invariants}\label{sec:invariants}

In this section we prove Theorem \ref{thm:main2}.
In contrast to the previous sections, we focus on unramified coverings.

\subsection{$\Z_p^d$-coverings of bouquets}\label{ss:bou_rev}

We begin with a basic discussion parallel to \cite{Kata_36}.
Let $\Gamma \simeq \Z_p^d$ with $d \geq 2$.
Let $X$ be a bouquet with edges
\[
\bE_X = \{e_1, \dots, e_t, \ol{e_1}, \dots, \ol{e_t}\}.
\]
Let $\alpha: \bE_X \to \Gamma$ be a voltage assignment.
To ensure that the derived graphs are connected, we assume that the image of $\alpha$ is not contained in a proper closed subgroup of $\Gamma$.
By identifying $\cL_{\alpha}$ with the scalar, we have
\begin{align*}
\cL_{\alpha}
&= \sum_{i=1}^t (1 - \alpha(e_i)) + \sum_{i=1}^t (1 - \alpha(\ol{e_i}))\\
&= \sum_{i=1}^t (1 - \alpha(e_i)) + \sum_{i=1}^t (1 - \alpha(e_i)^{-1}).
\end{align*}
Set
\[
h_{\alpha}
= (1 - \alpha(e_1)) + \dots + (1 - \alpha(e_t))
\in \Z[\Gamma],
\]
which implicitly depends on the orientation of $X$.
Then we have
\[
\cL_{\alpha} 
= h_{\alpha} + \iota(h_{\alpha}),
\]
where $\iota$ denotes the involution of $\Z_p[[\Gamma]]$ that inverts every group element.
Therefore, by Theorem \ref{thm:main1_r}, the derived graph $X_{\infty}/X$ satisfies
\[
\lambda(X_{\infty}/X) = l_0(h_{\alpha} + \iota(h_{\alpha}))
\AND
\mu(X_{\infty}/X) = m_0(h_{\alpha} + \iota(h_{\alpha})).
\]
We now give a name to the elements of the form $h_{\alpha}$.

\begin{defn}\label{defn:adm}
We say an element $h$ of $\Z[\Gamma]$ admissible if
\begin{itemize}
\item[(a)]
$\aug(h) = 0$,
\item[(b)]
the coefficients of non-identity elements are all non-positive, and
\item[(c)]
$h \not \in \Z[\Gamma']$ for any proper subgroup $\Gamma'$ of $\Gamma$.
\end{itemize}
\end{defn}

Then it is clear that $h_{\alpha}$ is admissible for any voltage assignment $\alpha$, and conversely any admissible $h$ is of the form $h_{\alpha}$ for some bouquet and some voltage assignment $\alpha$.
Therefore, the possible pairs of $(\lambda(X_{\infty}/X), \mu(X_{\infty}/X))$ with $X$ a bouquet and $X_{\infty}/X$ is an unramified $\Gamma$-covering are precisely the possible pairs of
\[
(l_0(h + \iota(h)), m_0(h + \iota(h)))
\]
with $h \in \Z[\Gamma]$ admissible.

In the next subsection we prove the following key proposition.

\begin{prop}\label{prop:constr_key}
Let $p$ be a prime number and $\Gamma \simeq \Z_p^d$ with $d \geq 2$.
Then for any non-negative integer $\lambda \geq 0$, there exists an element $\ol{h} \in \F_p[[\Gamma]]$ with $\aug(\ol{h}) = 0$ such that $l_0(\ol{h} + \iota(\ol{h})) = \lambda$.
\end{prop}

We note that this fails if $d = 1$, where $l_0(\ol{h} + \iota(\ol{h}))$ is necessarily positive and even (see \cite{Kata_36}).

Let us deduce Theorem \ref{thm:main2}, assuming this proposition.

\begin{proof}[Proof of Theorem \ref{thm:main2}]
Take any pair $(\lambda, \mu)$ of non-negative integers.
By Proposition \ref{prop:constr_key}, we can take an element $\ol{h} \in \F_p[\Gamma]$ such that $\aug(\ol{h}) = 0$ and $l_0(\ol{h} + \iota(\ol{h})) = \lambda$.

We claim that there is a lift of $\ol{h}$ to $h \in \Z[\Gamma]$ such that $h$ is admissible in the sense of Definition \ref{defn:adm}.
For this, we first take any lift $h \in \Z[\Gamma]$ and modify it by using the liberty of adding $p \Z[\Gamma]$ to $h$.
Condition (b) can be satisfied by subtracting suitable elements of the form $p \gamma$ with $\gamma \in \Gamma$.
Condition (c) can be satisfied by subtracting $p \sigma_1, \dots, p \sigma_d$, where $\sigma_1, \dots, \sigma_d$ is a basis of $\Gamma$.
Finally, since $\aug(h) \in p\Z$, condition (a) can be satisfied by subtracting $\aug(h)$.

Note that $m_0(h + \iota(h)) = 0$ (i.e., $h + \iota(h)$ is not divisible by $p$) since $\ol{h} + \iota(\ol{h}) \neq 0$.
Then the element $p^{\mu} h \in \Z[\Gamma]$ is again admissible and satisfies
\[
l_0(p^{\mu} h + \iota(p^{\mu} h)) = \lambda
\AND
m_0(p^{\mu} h + \iota(p^{\mu} h)) = \mu.
\]
By taking the bouquet $X$ and the voltage assignment $\alpha: \bE_X \to \Gamma$ such that $h_{\alpha} = p^{\mu} h$, the derived $\Gamma$-covering $X_{\infty}/X$ has the prescribed $\lambda$- and $\mu$-invariants.
\end{proof}

\subsection{$l_0$-invariants with $\F_p$-coefficients}\label{ss:l_0_res}

The aim of this subsection is to prove Proposition \ref{prop:constr_key}.
Let $\Gamma \simeq \Z_p^d$ with $d \geq 2$.

\begin{lem}\label{lem:key_constr1}
For $h_1, h_2 \in \F_p[[\Gamma]]$, setting $h = h_1(h_2 + \iota(h_2))$, we have
\[
l_0(h + \iota(h)) = l_0(h_1 + \iota(h_1)) + l_0(h_2 + \iota(h_2))
\]
as long as $h_1 + \iota(h_1) \neq 0$ and $h_2 + \iota(h_2) \neq 0$.
Moreover, if $\aug(h_1) = 0$ or $\aug(h_2) = 0$, then $\aug(h) = 0$ holds.
\end{lem}

\begin{proof}
The first claim follows from 
\[
h + \iota(h) = (h_1 + \iota(h_1)) (h_2 + \iota(h_2))
\]
and the second from $\aug(h) = 2 \aug(h_1) \aug(h_2)$.
\end{proof}

Because of this lemma, we only have to construct values $0$ and $1$.
By considering $h$ of the form $a \sigma^{\alpha} \tau^{\beta} + b \sigma + c \tau - (a + b + c)$, we find the following examples.

\begin{prop}\label{prop:key_constr2}
Let $\sigma, \tau \in \Gamma$ be elements that can be extended to a basis of $\Gamma$.
\begin{itemize}
\item[(1)]
For $h = \sigma^3 \tau^2 + \sigma - \tau - 1$, we have $l_0(h + \iota(h)) = 0$.
\item[(2)]
For $h = \sigma^2 \tau^2 + \sigma - \tau - 1$, we have $l_0(h + \iota(h)) = 1$.
\end{itemize}
\end{prop}

\begin{proof}
(1)
We have
\[
h + \iota(h) 
= \sigma^3 \tau^2 + \sigma^{-3} \tau^{-2} + \sigma + \sigma^{-1} - \tau - \tau^{-1} - 2.
\]
Suppose, for contradiction, that there exists an element $\gamma \in \Gamma$ such that $\gamma - 1$ divides $h + \iota(h)$ in $\F_p[[\Gamma]]$.
Let $\lrang{\sigma, \tau} \simeq \Z_p^2$ be the closed subgroup of $\Gamma$.
We take a $\Z_p$-linear left inverse $\Gamma \twoheadrightarrow \lrang{\sigma, \tau}$ of the inclusion $\lrang{\sigma, \tau} \hookrightarrow \Gamma$, and write $\gamma \mapsto \sigma^{\alpha} \tau^{\beta}$ with $\alpha, \beta \in \Z_p$.
Then the element $\sigma^{\alpha} \tau^{\beta} - 1$ divides $h + \iota(h)$ in $\F_p[[\lrang{\sigma, \tau}]]$.
Since $h + \iota(h) \neq 0$, one of $\alpha, \beta$ is nonzero.

We first suppose that $\ord_p(\alpha) \geq \ord_p(\beta)$.
Then $\beta \neq 0$ and $(\sigma^{\alpha/\beta} \tau - 1) \mid (\sigma^{\alpha} \tau^{\beta} - 1)$, so the element $\sigma^{\alpha/\beta} \tau - 1$ divides $h + \iota(h)$ in $\F_p[[\lrang{\sigma, \tau}]]$.
This implies
\[
\sigma^{3 - 2 \alpha/\beta} + \sigma^{-3 + 2 \alpha/\beta} + \sigma + \sigma^{-1} - \sigma^{-\alpha/\beta} - \sigma^{\alpha/\beta} - 2 = 0
\]
in $\F_p[[\lrang{\sigma}]]$.
Since the third term $\sigma$ must be cancelled by another element, one of the other exponents
\[
3 - 2 \alpha/\beta, -3 + 2 \alpha/\beta, -1, -\alpha/\beta, \alpha/\beta, 0
\]
is equal to $1$ (as an element of $\Z_p$).
This implies $\alpha/\beta \in \{-1, 1, 2\}$.
However, in any case the displayed equation in $\F_p[[\lrang{\sigma}]]$ does not hold.
This is a contradiction.
When $\ord_p(\alpha) \leq \ord_p(\beta)$, we can also reach a contradiction in the same manner.

(2)
We have
\begin{align*}
h + \iota(h) 
& = \sigma^2 \tau^2 + \sigma^{-2} \tau^{-2} + \sigma + \sigma^{-1} - \tau - \tau^{-1} - 2\\
& = (1 - \sigma^{-1} \tau^{-1}) (\sigma^2 \tau^2 + \sigma \tau + \sigma - \tau - 1 - \sigma^{-1} \tau^{-1}).
\end{align*}
We can check that $l_0(\sigma^2 \tau^2 + \sigma \tau + \sigma - \tau - 1 - \sigma^{-1} \tau^{-1}) = 0$ in the same way as in (1).
\end{proof}

Now Proposition \ref{prop:constr_key} is established and thus the proof of Theorem \ref{thm:main2} is completed.

For explicit examples it is convenient to find simpler elements $h$.
The following remark deals with the case $\lambda \in \{2, 4, 6, \dots\}$ and is essentially due to \cite{Kata_36}.

\begin{rem}\label{rem:constr_even}
Let $\lambda$ be a positive even integer.
For any $\gamma \in \Gamma \setminus \Gamma^p$, the element
\[
h = \sum_{i=0}^{\lambda/2 - 1} (-1)^i \binom{\lambda}{i} \gamma^{\lambda/2 - i} + (-1)^{\lambda/2} \frac{1}{2} \binom{\lambda}{\lambda/2}
\]
satisfies $\aug(h) = 0$ and $l_0(h + \iota(h)) = \lambda$.
Indeed, since
\begin{align*}
\iota(h)
& = \sum_{i=0}^{\lambda/2 - 1} (-1)^i \binom{\lambda}{i} \gamma^{-\lambda/2 + i} + (-1)^{\lambda/2} \frac{1}{2} \binom{\lambda}{\lambda/2}\\
& = \sum_{i=\lambda/2 + 1}^{\lambda} (-1)^{\lambda - i} \binom{\lambda}{\lambda - i} \gamma^{-\lambda/2 + (\lambda - i)} + (-1)^{\lambda/2} \frac{1}{2} \binom{\lambda}{\lambda/2}\\
& = \sum_{i=\lambda/2 + 1}^{\lambda} (-1)^i \binom{\lambda}{i} \gamma^{\lambda/2 - i} + (-1)^{\lambda/2} \frac{1}{2} \binom{\lambda}{\lambda/2},
\end{align*}
we have
\[
h + \iota(h)
= \sum_{i=0}^{\lambda} (-1)^i \binom{\lambda}{i} \gamma^{\lambda/2 - i}
= \gamma^{-\lambda/2} (\gamma - 1)^{\lambda}.
\]
Note that this description of $h$ is independent of $p$.
See Table \ref{tab:l_even}.
\end{rem}

\begin{table}[htbp]
\centering
\begin{tabular}{c | @{\hspace{1em}} c c}
\toprule
$l_0(h + \iota(h))$ & $h$ & $h$ when $p =2$ \\
\midrule
$2$ & $\gamma - 1$ & $\gamma + 1$\\
$4$ & $\gamma^2 - 4 \gamma + 3$ & $\gamma^2 + 1$\\
$6$ & $\gamma^3 - 6 \gamma^2 + 15 \gamma - 10$ & $\gamma^3 + \gamma$\\
$8$ & $\gamma^4 - 8\gamma^3 + 28 \gamma^2 - 56 \gamma + 35$ & $\gamma^4 + 1$\\
$10$ & $\gamma^5 - 10 \gamma^4 + 45 \gamma^3 - 120 \gamma^2 + 210 \gamma - 126$ & $\gamma^5 + \gamma^3$\\
$12$ & $\gamma^6 - 12 \gamma^5 + 66 \gamma^4 - 220 \gamma^3 + 495 \gamma^2 - 792 \gamma + 462$ & $\gamma^6 + \gamma^2$\\
\bottomrule
\end{tabular}
\caption{Elements with $l_0(h + \iota(h)) \in \{2, 4, 6, 8, 10, 12\}$ for any $p$}
\label{tab:l_even}
\end{table}

The following remark deals with the case $\lambda \in \{0, 1, 3, 5, 7, \dots\}$.

\begin{rem}\label{rem:constr_odd}
\begin{itemize}
\item
When $p \geq 3$, the element $h = \sigma + \tau - 2$ satisfies $l_0(h + \iota(h)) = 0$.
Note that when $p = 2$, the same element satisfies $l_0(h + \iota(h)) = 2$.
\item
When $p = 2$, the element $h = \sigma \tau + \sigma + \tau + 1$ satisfies $l_0(h + \iota(h)) = 3$.
Indeed,
\[
h + \iota(h) 
= \sigma \tau + \sigma^{-1} \tau^{-1} + \sigma + \sigma^{-1} + \tau + \tau^{-1}
= (1 - \sigma^{-1} \tau^{-1}) (\sigma - 1)(\tau -1).
\]
\item
When $p \geq 5$, taking any $a \in \F_p \setminus \{0, 1, -1\}$, the element $h = a \sigma \tau + \sigma - \tau - a$ satisfies $l_0(f) = 1$.
Indeed,
\begin{align*}
h + \iota(h) 
& = a \sigma \tau + a\sigma^{-1} \tau^{-1} + \sigma + \sigma^{-1} - \tau - \tau^{-1} - 2a\\
& = (1 - \sigma^{-1} \tau^{-1})(a \sigma \tau + \sigma - \tau - a)
\end{align*}
and we have $l_0(a \sigma \tau + \sigma - \tau - a) = 0$.
\end{itemize}
\end{rem}

Examples when $p = 2$ are listed in Table \ref{tab:l_odd}.
We use Proposition \ref{prop:key_constr2} for $\lambda = 0, 1$ and Remark \ref{rem:constr_odd} for $\lambda = 3$.
For $\lambda = 5, 7, 9, 11$, we decompose $\lambda = 3 + (\lambda - 3)$ and apply Lemma \ref{lem:key_constr1} to the case $\lambda = 3$ and the case $2, 4, 6, 8$ obtained in Table \ref{tab:l_even} with $\gamma = \sigma \tau$.

\begin{table}[htbp]
\centering
\begin{tabular}{c @{\hspace{1em}} c}
\toprule
$l_0(h + \iota(h))$ & $h$ \\
\midrule
$0$ & $\sigma^3 \tau^2 + \sigma + \tau + 1$\\
$1$ & $\sigma^2 \tau^2 + \sigma + \tau + 1$\\
$3$ & $\sigma \tau + \sigma + \tau + 1$\\
$5$ & $\sigma^2 \tau^2 +  \sigma^2 \tau + \sigma \tau^2 + \sigma + \tau + 1$\\
$7$ & $\sigma^3 \tau^3 + \sigma^3 \tau^2 + \sigma^2 \tau^3 + \sigma^2 \tau + \sigma \tau^2 + \sigma \tau$\\
$9$ & $\sigma^4 \tau^4 + \sigma^4 \tau^3 + \sigma^3 \tau^4 + \sigma^3 \tau^2 +  \sigma^2 \tau^3 + \sigma^2 \tau + \sigma \tau^2 + \sigma + \tau + 1$\\
$11$ & $\sigma^5 \tau^5 + \sigma^5 \tau^4 + \sigma^4 \tau^5 + \sigma^4 \tau^3 +  \sigma^3 \tau^4 + \sigma^3 \tau^3$\\
\bottomrule
\end{tabular}
\caption{Elements with $l_0(h + \iota(h)) \in \{0, 1, 3, 5, 7, 9, 11\}$ for $p = 2$}
\label{tab:l_odd}
\end{table}

\subsection{Numerical examples}\label{ss:examples_bouquets}

We focus on $p = 2$ and $d = 2$.
We take a basis $\sigma, \tau$ of $\Gamma$.
By Tables \ref{tab:l_even} (with $\gamma = \sigma$) and \ref{tab:l_odd}, for each $\lambda = 0, 1, 2, \dots, 12$, we obtain explicit examples of voltage assignments realizing $\lambda$ with $\mu = 0$ (recall that $\mu$ can be made arbitrary by multiplying the number of edges as in the proof of Theorem \ref{thm:main2}).
The resulting examples are listed in Table \ref{tab:voltages}.
Note that ``$\sigma, \sigma$'' and ``$\tau, \tau$'' are added to ensure that the derived graphs are connected, i.e., to make the element $h$ admissible.

\begin{table}[htbp]
\centering

\begin{minipage}{0.4\linewidth}
\centering
\begin{tabular}{c |@{\hspace{1em}} c}
\toprule
$\lambda$ & voltage $\alpha(e_1),\dots,\alpha(e_t)$ \\
\midrule
2 & $\sigma, \tau, \tau$ \\
4 & $\sigma^2, \sigma, \sigma, \tau, \tau$ \\
6 & $\sigma^3, \sigma, \tau, \tau$ \\
8 & $\sigma^4, \sigma, \sigma, \tau, \tau$ \\
10 & $\sigma^5, \sigma^3, \tau, \tau$\\
12 & $\sigma^6, \sigma^2, \sigma, \sigma, \tau, \tau$\\
\bottomrule
\end{tabular}
\end{minipage}
\hfill
\begin{minipage}{0.55\linewidth}
\centering
\begin{tabular}{c |@{\hspace{1em}} c}
\toprule
$\lambda$ & voltage $\alpha(e_1),\dots,\alpha(e_t)$ \\
\midrule
0 & $\sigma^3 \tau^2, \sigma, \tau$ \\
1 & $\sigma^2 \tau^2, \sigma, \tau$ \\
3 & $\sigma \tau, \sigma, \tau$ \\
5 & $\sigma^2 \tau^2, \sigma^2 \tau, \sigma \tau^2, \sigma, \tau$ \\
7 & $\sigma^3 \tau^3, \sigma^3 \tau^2, \sigma^2 \tau^3, \sigma^2 \tau, \sigma \tau^2, \sigma \tau$ \\
9 & $\sigma^4 \tau^4, \sigma^4 \tau^3, \sigma^3 \tau^4, \sigma^3 \tau^2, \sigma^2 \tau^3, \sigma^2 \tau, \sigma \tau^2,  \sigma, \tau$\\
11 & $\sigma^5 \tau^5, \sigma^5 \tau^4, \sigma^4 \tau^5, \sigma^4 \tau^3, \sigma^3 \tau^4, \sigma^3 \tau^3$\\
\bottomrule
\end{tabular}
\end{minipage}

\caption{Voltage assignments with prescribed $\lambda$ with $\mu=0$ for $p = 2$}
\label{tab:voltages}
\end{table}

For example, the derived graphs for the voltage assignment for $\lambda = 3$ is as follows.
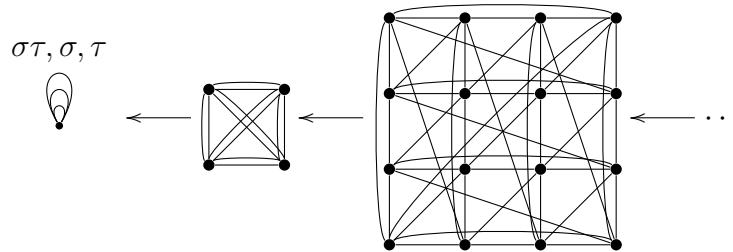
\[
\xymatrix{
\begin{tikzpicture}[baseline]
\node[circle,fill,inner sep=1pt] (a) {};
\foreach \len in {15,30}
\path (a) edge[loop above, looseness=\len, out=60, in=120] (a);
\path (a) edge[loop above, looseness=45, out=60, in=120] node[above] {$\sigma \tau, \sigma, \tau$} (a);
\end{tikzpicture}
&
\begin{tikzpicture}[baseline={(current bounding box.center)}]
\foreach \x in {0,1}
{
  \foreach \y in {0,1}
  {
    \node[circle,fill,inner sep=1.5pt]
      (v\x\y) at (\x,\y) {};
  }
}
\foreach \y in {0,1}
{
  \foreach \x/\xx in {0/1}
  {
    \draw (v\x\y) -- (v\xx\y);
  }
}
\foreach \y in {0,1}
{
  \draw (v1\y) to[out=135,in=45,looseness=0.2] (v0\y);
}
\foreach \x in {0,1}
{
  \foreach \y/\yy in {0/1}
  {
    \draw (v\x\y) -- (v\x\yy);
  }
}
\foreach \x in {0,1}
{
  \draw (v\x1) to[out=225,in=135,looseness=0.2] (v\x0);
}
\draw (v00) -- (v11);
\draw (v01) -- (v10);
\draw (v11) to[out=180,in=90,looseness=0.2] (v00);
\draw (v10) to[out=90,in=0,looseness=0.2] (v01);
\end{tikzpicture}
\ar[l]
&
\begin{tikzpicture}[baseline={(current bounding box.center)}]
\foreach \x in {0,1,2,3}
{
  \foreach \y in {0,1,2,3}
  {
    \node[circle,fill,inner sep=1.5pt]
      (v\x\y) at (\x,\y) {};
  }
}
\foreach \y in {0,1,2,3}
{
  \foreach \x/\xx in {0/1,1/2,2/3}
  {
    \draw (v\x\y) -- (v\xx\y);
  }
}
\foreach \y in {0,1,2,3}
{
  \draw (v3\y) to[out=135,in=45,looseness=0.2] (v0\y);
}
\foreach \x in {0,1,2,3}
{
  \foreach \y/\yy in {0/1,1/2,2/3}
  {
    \draw (v\x\y) -- (v\x\yy);
  }
}
\foreach \x in {0,1,2,3}
{
  \draw (v\x3) to[out=225,in=135,looseness=0.2] (v\x0);
}
\foreach \x/\xx in {0/1,1/2,2/3,3/0}
{
  \foreach \y/\yy in {0/1,1/2,2/3}
  {
    \draw (v\x\y) -- (v\xx\yy);
  }
}
\foreach \x/\xx in {0/1,1/2,2/3}
{
  \draw (v\x3) -- (v\xx0);
}
\draw (v33) to[out=180,in=90,looseness=0.2] (v00);
\end{tikzpicture}
\ar[l] 
& \cdots \ar[l]}
\]

For $\lambda = 0, 1, 2, \dots, 8$ and for $n = 0, 1, 2, \dots, 6$, by using SageMath, we computed $\ord_p(\kappa_{X_n})$ of the derived $\Z_p^2$-covering $X_{\infty}/X$.
The results are listed in Table \ref{tab:l_all_comp}.
The table also shows the suggested asymptotic formula; theoretically we verified the coefficients of $np^n$, but the other coefficients are not confirmed.

\begin{table}[htbp]
\centering
\begin{tabular}{c |@{\hspace{1em}} c c c c c c c @{\hspace{1em}} c}
\toprule
$\lambda$ & \swid{$0$} & \swid{$1$} & \swid{$2$} & \swid{$3$} & \swid{$4$} & \swid{$5$} & \swid{$6$} & Suggested formula\\
\midrule
2 & 0 & 5 & 21 & 55 & 137 & 331 & 781 & $2np^n+2n+1$ ($n \geq 2$)\\
4 & 0 & 8 & 36 & 90 & 224 & 550 & 1324 & $4np^n - 4p^n + 6n + 8$ ($n \geq 2$)\\
6 & 0 & 8 & 34 & 120 & 294 & 724 & 1762 & $6np^n - 10p^n + 14n + 14$ ($n \geq 3$)\\
8 & 0 & 8 & 34 & 112 & 318 & 844 & 2138 & $8np^n - 16p^n + 14n + 6$ ($n \geq 2$)\\
\bottomrule
\end{tabular}
\centering
\begin{tabular}{c |@{\hspace{1em}} c c c c c c c @{\hspace{1em}}  c}
\toprule
$\lambda$ & \swid{$0$} & \swid{$1$} & \swid{$2$} & \swid{$3$} & \swid{$4$} & \swid{$5$} & \swid{$6$} & Suggested formula\\
\midrule
0 & 0 & 5 & 21 & 63 & 133 & 275 & 561 & $9p^n - 2n - 3$ ($n \geq 3$)\\
1 & 0 & 5 & 19 & 65 & 161 & 373 & 833 & $np^n + 15/2p^n - 4n - 7$ ($n \geq 3$)\\
3 & 0 & 7 & 35 & 111 & 295 & 719 & 1671 & $3np^n+9p^n-8n-9$ ($n \geq 1$)\\
5 & 0 & 8 & 34 & 134 & 358 & 902 & 2166 & $5np^n+5p^n-16n+22$ ($n \geq 3$)\\
7 & 0 & 10 & 50 & 154 & 446 & 1162 & 2838 & $7np^n + 4p^n - 20n + 14$ ($n \geq 3$)\\
\bottomrule
\end{tabular}
\caption{The values of $\ord_p(\kappa_{X_n})$ for the voltage assignments given in Table \ref{tab:voltages}}
\label{tab:l_all_comp}
\end{table}

\section*{Acknowledgments}

This work is supported by JSPS KAKENHI Grant Number 26K16965.

{
\bibliographystyle{abbrv}
\bibliography{biblio}
}

\end{document}